\documentclass[12pt,a4paper]{elsarticle}
\usepackage{etex}
\usepackage[utf8]{inputenc}
\usepackage[TS1,T1]{fontenc}
\usepackage{xspace}
\usepackage[a4paper,tmargin=3truecm,bmargin=3truecm,rmargin=1.5truecm,lmargin=1.5truecm]{geometry}

\usepackage[all]{xy}

\usepackage[english]{babel}
\usepackage{lmodern}

\usepackage{enumitem}
\newlist{enum-hypothesis}{enumerate}{1}

\usepackage{amsfonts,amssymb,amsmath,mathtools,slashed}
\usepackage[thmmarks,amsmath]{ntheorem}

\usepackage{hyperref}

\newtheorem{theorem}{Theorem}[section]
\newtheorem{proposition}[theorem]{Proposition}
\newtheorem{lemma}[theorem]{Lemma}
\newtheorem{corollary}[theorem]{Corollary}

\theoremstyle{plain}
\theoremsymbol{$\square$}
\theoremstyle{plain}
\theorembodyfont{\upshape}
\newtheorem{remark}[theorem]{Remark}

\theoremsymbol{}
\theoremstyle{break}
\theorembodyfont{\slshape}

\theoremheaderfont{\scshape}
\theorembodyfont{\upshape} 
\theoremstyle{nonumberplain}
\theoremseparator{} 
\theoremsymbol{\rule{1.3ex}{1.3ex}} 
\newtheorem{proof}{Proof}

\hypersetup{
plainpages=false,
colorlinks=true,
linkcolor=black, 
anchorcolor=black, 
citecolor=black, 
urlcolor=black, 
menucolor=black, 
filecolor=black, 
bookmarksopen=true,
bookmarksnumbered=true}

\numberwithin{equation}{section}

\hypersetup{
pdfauthor={Thierry Masson, Bruno Iochum},
pdftitle={???},
pdfsubject={},
pdfcreator={pdflatex},
pdfproducer={pdflatex},
pdfkeywords={}
}

\newcommand\R{{\mathbb R}}
\newcommand\C{{\mathbb C}}
\newcommand\N{{\mathbb N}}
\newcommand\Z{{\mathbb Z}}
\newcommand{\TT}{\mathbb{T}}
\newcommand{\gS}{\mathbb{S}}

\newcommand{\BB}{\mathbb{B}}

\renewcommand{\H}{\mathcal{H}}

\newcommand{\cS}{\mathcal{S}}

\newcommand{\vc}{\vcentcolon =}   
\newcommand{\cv}{=\vcentcolon }   
\newcommand{\bbbone}{{\text{\usefont{U}{dsss}{m}{n}\char49}}}

\newcommand{\pmu}{\partial_\mu}  
\newcommand{\pnu}{\partial_\nu}   
\newcommand{\prho}{\partial_\rho}  

\newcommand{\nmu}{\nabla_\mu}   
\newcommand{\nnu}{\nabla_\nu}   
\newcommand{\nrho}{\nabla_\rho}   

\newcommand{\mul}{{\bf m}}    

\newcommand{\dmu}{\delta_\mu}
\newcommand{\dnu}{\delta_\nu}

\newcommand{\btau}{\bar{\tau}}
\newcommand{\bepsilon}{\bar{\epsilon}}
\newcommand{\bpartial}{\bar{\partial}}

\newcommand{\hnabla}{\widehat{\nabla}}
\newcommand{\hnmu}{\hnabla_\mu}
\newcommand{\hnnu}{\hnabla_\nu}
\newcommand{\hdelta}{\widehat{\delta}}
\newcommand{\hdmu}{\hdelta_\mu}
\newcommand{\hdnu}{\hdelta_\nu}

\newcommand{\NCtr}{\mathfrak{t}}

\DeclareMathOperator{\Tr}{Tr}	   
\DeclareMathOperator{\tr}{tr}	   
\DeclareMathOperator{\End}{End}	 
\DeclareMathOperator{\spec}{sp}   

\newcommand{\dd}{\text{d}}
\newcommand{\dvolg}{\text{dvol}_g}
\newcommand{\dvolh}{\text{dvol}_h}

\newcommand{\Modular}{{\boldsymbol{\Delta}}}

\DeclarePairedDelimiter\abs{\lvert}{\rvert}
\DeclarePairedDelimiter\norm{\lVert}{\rVert}

\newcounter{mnotecount}[section]
\renewcommand{\themnotecount}{\thesection.\arabic{mnotecount}}
\newcommand{\mnote}[1]
{\protect{\stepcounter{mnotecount}}$^{\mbox{\footnotesize
$
\bullet$\themnotecount}}$ \marginpar{
\raggedright\tiny\em
$\!\!\!\!\!\!\,\bullet$\themnotecount: #1} }

\allowdisplaybreaks

\newcommand{\hlterm}[1]{\boldsymbol{#1}}

\renewcommand\hlterm[1]{#1}

\begin{document}

\title{Heat asymptotics for nonminimal Laplace type operators\\ and application to noncommutative tori}
\author[CPT]{B. Iochum}
\ead{bruno.iochum@cpt.univ-mrs.fr}
\author[CPT]{T. Masson}
\ead{thierry.masson@cpt.univ-mrs.fr}

\address[CPT]{Centre de Physique Théorique\\
Aix Marseille Univ, Université de Toulon, CNRS, CPT, Marseille, France}

\begin{abstract}
Let $P$ be a Laplace type operator acting on a smooth hermitean vector bundle $V$ of fiber $\C^N$ over a compact Riemannian manifold given locally by  $P= - [g^{\mu\nu} u(x)\partial_\mu\partial_\nu + v^\nu(x)\partial_\nu + w(x)]$ where $u,\,v^\nu,\,w$ are $M_N(\C)$-valued functions with $u(x)$ positive and invertible. For any $a \in \Gamma(\text{End}(V))$, we consider the asymptotics $\Tr (a e^{-tP}) \underset{t \downarrow 0^+}{\sim} \,\sum_{r=0}^\infty a_r(a, P)\,t^{(r-d)/2}$ where the coefficients $a_r(a, P)$ can be written locally as $a_r(a, P)(x) = \tr[a(x) \mathcal{R}_r(x)]$.\\
The computation of $\mathcal{R}_2$ is performed opening the opportunity to calculate the modular scalar curvature for  noncommutative tori.
\end{abstract}

\begin{keyword}
Heat kernel\sep nonminimal operator\sep asymptotic heat trace\sep Laplace type operator, scalar curvature, noncommutative torus

\PACS 11.15.-q \sep 04.62.+v

\MSC 58J35 \sep 35J47 \sep 81T13 \sep 46L87
\end{keyword}

\maketitle

\section{Introduction}

As in \cite{IochMass17a}, we consider a $d$-dimensional compact Riemannian manifold $(M, g)$ without boundary, together with a nonminimal Laplace type operator $P$ on a smooth hermitean vector bundle $V$ over $M$ of fiber $\C^N$ written locally as
\begin{align}
\label{eq-def-P-uvw}
P \vc- [\,g^{\mu\nu} u(x)\pmu\pnu + v^\nu(x)\pnu + w(x)\,].
\end{align}
Here $u(x) \in M_N(\C)$ is a positive and invertible matrix valued function and $v^\nu,\, w$ are $M_N(\C)$ matrices valued functions. The operator is expressed in a local trivialization of $V$ over an open subset of $M$ which is also a chart on $M$ with coordinates $(x^\mu)$. This trivialization is such that the adjoint for the hermitean metric corresponds to the adjoint of matrices and the trace on endomorphisms on $V$ becomes the usual trace $\tr$ on matrices. 

For any $a \in \Gamma(\End(V))$, we consider the asymptotics of the heat-trace
\begin{align}
\label{heat-trace-asympt}
\Tr (a e^{-tP}) \underset{t \downarrow 0^+}{\sim} \,\sum_{r=0}^\infty a_r(a, P)\,t^{(r-d)/2}.
\end{align}
where $\Tr$ is the operator trace. Each coefficient $a_r(a, P)$ can be written as
\begin{equation}
\label{eq-araP-dvol}
a_r(a, P) = \int_M a_r(a, P)(x) \, \dvolg(x)
\end{equation} 
where $\dvolg(x) \vc \abs{g}^{1/2} \dd x$ with $\abs{g} \vc \det(g_{\mu\nu})$. The functions $a_r(a, P)(x)$ can be evaluated (various techniques exist for that) and give expressions of the form
\begin{equation*}
a_r(a, P)(x) = \tr[a(x) \mathcal{R}_r(x)],
\end{equation*}
where $\tr$ is the trace on matrices and $\mathcal{R}_r$ is a (local) section of $\End(V)$. The local section $\mathcal{R}_r$ of $\End(V)$ is uniquely defined by
\begin{equation}
\label{eq-Rr-trace-choice}
a_r(a, P) = \varphi(a \mathcal{R}_r),
\end{equation}
where $\varphi(a) \vc \int_M  \tr[a(x)] \, \dvolg(x)$ is the natural combined trace on the algebra of sections of $\End(V)$ associated to $(M,g)$ (the integral) and $V$ (the matrix trace). The choice of this trace is not unique, and changing $\varphi$ changes $\mathcal{R}_r$. For instance, since $M$ is compact, one can normalize the integral so that the total volume of $M$ is $1$, and also the matrix trace such that the trace of the identity matrix is $1$. In that case, denoted by $\bbbone$ the identity operator in $\Gamma(\End(V))$, the new combined trace $\varphi_0$ satisfies $\varphi_0(\bbbone) = 1$. In Section~\ref{sec-applications-NCT} $\varphi_0$ plays an important role since it corresponds to the unique normalized trace on the noncommutative torus algebra. Another possibility for the choice of $\varphi$ is to use a Riemannian metric on $M$ which is not the tensor $g$ in $P$, see Remark~\ref{rem-metrics}.

The aim of this paper is to present a way to compute $\mathcal{R}_r$ by adapting the techniques developed in \cite{IochMass17a}. These techniques were strongly motivated by a need in physics for explicit computations of $a_r(\bbbone,P)$, see for instance \cite{Avra04a, AvraBran01a} and the reference in \cite{IochMass17a} for the existing results on the mathematical side. The idea behind the computation of $\mathcal{R}_2$ is to extract the real matrix content of the coefficient $a_2$ which is related to the scalar curvature of the manifold $M$.

In Section~\ref{sec-method-results}, two formulas are provided for $\mathcal{R}_2(x)$, both in local coordinates (Theorem \ref{thm-R-uvw}) and in a covariant way (Theorem \ref{thm-R-upq}) in arbitrary dimension and detailed in low dimensions. In Section~\ref{sec-direct-applications}, some direct applications are also provided, for instance to a conformal like transformed Laplacian. Section~\ref{sec-details-computations} is devoted to the details of the computations (see also the ancillary \texttt{Mathematica}  \cite{Wolf17a} notebook file \cite{IochMass17b}).

In Section~\ref{sec-applications-NCT}, another applications are given in noncommutative geometry. Namely, we compute the conformally deformed scalar curvature of  rational noncommutative tori (NCT). Since at rational values $\theta = p/q$ of the deformation parameter, the algebras of the NCT are isomorphic to the continuous sections of a bundle over the ordinary tori with fiber in $M_q(\C)$, they fit perfectly with our previous framework. The irrational case has been widely studied in \cite{ConnTret11a, ConnMosc14a, FathKhal11a, FathKhal12a, DabSit13, FathKhal13a, ALNP, Sitarz14, Fath15a, DabSit15, Liu15a, Sade16a, ConnesFath16}. The results presented in these papers can be written without explicit reference to the parameter $\theta$. In the rational case, our results confirm this property. Moreover, our method gives an alternative which avoids the theory of pseudodifferential calculus on the noncommutative tori introduced by Connes and Tretkoff \cite{Connes80,ConnTret11a}, see also \cite{LeschMosco16}. In Appendix~\ref{sec-comparison-NCT}, in order to confirm the results in \cite[Theorem~5.2]{FathKhal11a} and \cite[Theorem~5.4]{FathKhal13a}, we perform the change of variables from $u$ to $\ln(u)$ and the change of operators from the left multiplication by $u$ to the conjugation by $u$, formalized as a rearrangement lemma (Lemma~\ref{lem-rearrangement-lemma}).

\section{The method and the results}
\label{sec-method-results}

In \cite{IochMass17a}, the computation was done for the special case $a = \bbbone$, for which a lot of simplifications can be used under the trace. We now show that the method described there can be adapted almost without any change to compute the quantities $\mathcal{R}_r$. Moreover, we present the method in a way that reduces the number of steps in the computations, using from the beginning covariant derivatives on the vector bundle $V$.

\subsection{Notations and preliminary results}

In order to start with the covariant form of $P$ (see \cite[Section~A.4]{IochMass17a}), let us introduce the following notations. We consider a covariant derivative $\nmu \vc \pmu + \eta(A_\mu)$, where $\eta$ is the representation of the Lie algebra of the gauge group of $V$ on any associated vector bundles (mainly $V$ and $\End(V)$ in the following). Let 
\begin{align*}
\alpha_\mu &\vc g_{\rho\sigma}(\pmu g^{\rho\sigma}),
&
\alpha^\mu &\vc g^{\mu\nu} \alpha_\nu = g^{\mu\nu} g_{\rho\sigma}(\pnu g^{\rho\sigma}),
&
\beta_\mu &\vc g_{\mu\sigma}(\prho g^{\rho\sigma}),
&
\beta^\mu &\vc g^{\mu\nu} \beta_\nu = \pnu g^{\mu\nu}.
\end{align*}
The covariant form of $P$ associated to $\nabla$ (see \cite[eq.~(A.11)]{IochMass17a}) is given by
\begin{align}
P &= -(\abs{g}^{-1/2} \nmu \abs{g}^{1/2} g^{\mu\nu} u \nnu + p^\mu \nmu +q)
\nonumber
\\
&= - g^{\mu\nu} u \nmu \nnu - \big( p^\nu + g^{\mu\nu} (\nmu u) - [\tfrac{1}{2} \alpha^\nu - \beta^\nu ] u \big) \nnu - q,
\label{eq-def-P-upq}
\end{align}
where the last equality is obtained using $g^{\mu\nu} \tfrac{1}{2} \pmu \ln \abs{g} + \pmu g^{\mu\nu} = - \big[\tfrac{1}{2} \alpha^\nu - \beta^\nu \big]$. Here, $u$ is as before, and $p^\mu, \, q$ are as $v^\mu,\, w$ from \eqref{eq-def-P-uvw}, except that they transform homogeneously in a change of trivialization of $V$. All these (local) functions are $M_N(\C)$-valued (as local sections of $\End(V)$), so that $\eta$ is the adjoint representation:
\begin{align*}
\nmu u = \pmu u + [A_\mu, u],
\qquad
\nmu p^\nu = \pmu u + [A_\mu, p^\nu],
\qquad
\nmu q = \pmu u + [A_\mu, q].
\end{align*}

Let us introduce the total covariant derivative $\hnabla_\mu$, which combines $\nmu$ with the Levi-Civita covariant derivative induced by the metric $g$. It satisfies
\begin{align*}
\hnabla_\mu a^\nu 
&= \nmu a^\nu + \Gamma_{\mu \rho}^\nu a^\rho
= \pmu a^\nu + [A_\mu, a^\nu] + \Gamma_{\mu \rho}^\nu a^\rho,
&
\hnabla_\mu g^{\alpha \beta} &= 0,
\\
\hnabla_\mu b_\nu 
&=  \nmu b_\nu - \Gamma_{\mu\nu}^\rho b_\rho
= \pmu b_\nu + [A_\mu, b_\nu] - \Gamma_{\mu\nu}^\rho b_\rho,
&
\hnabla_\mu g_{\alpha \beta} &= 0,
\end{align*}
for any $\End(V)$-valued tensors $a^\nu$ and $b_\nu$, where $\Gamma_{\mu \rho}^\nu$ are the Christoffel symbols of $g$. Let us store the following relations:
\begin{align}
\hnabla_\mu u & = \nmu u,
\label{eq-nablaLCu}
\\
g^{\mu\nu} \hnabla_\mu \hnabla_\nu u 
&=
g^{\mu\nu} ( \nmu \nnu u - \Gamma^\rho_{\mu\nu} \nrho u )
=
g^{\mu\nu} \nmu \nnu u - [ \tfrac{1}{2} \alpha^\mu - \beta^\mu] \nmu u,
\label{eq-laplacianLCu}
\\
\hnabla_\mu p^\mu
&= \nmu p^\mu - \tfrac{1}{2} g_{\alpha\beta} (\pmu g^{\alpha\beta}) p^\mu
= \nmu p^\mu - \tfrac{1}{2} \alpha_\mu p^\mu.
\nonumber
\end{align}
Using $\tfrac{1}{2} \alpha^\rho - \beta^\rho = g^{\mu\nu} \Gamma_{\mu\nu}^\rho$, one then has
\begin{align}
\label{eq-P-upq-hatnabla}
P &= - ( g^{\mu\nu}  \hnmu u \hnnu  + p^\nu \hnnu +q)
= - g^{\mu\nu} u \hnmu \hnnu - [ p^\nu + g^{\mu\nu} (\hnmu u) ] \hnnu - q.
\end{align}
Notice that in these expressions the total covariant derivative $\hnabla_\nu$ (which is the first to act) will never apply to a tensor valued section of $V$, so that it could be reduced to the covariant derivative $\nabla_\nu$.

The writing of $P$ in terms of a covariant derivative $\nabla$ is of course not unique:

\begin{proposition}
\label{prop-change-connection-pq}
Let $\nmu' = \nmu + \eta(\phi_\mu)$ be another covariant derivative on $V$. Then
\begin{align}
\label{eq-P-up'q'-hatnabla}
P &= - ( g^{\mu\nu}  \hnmu' u \hnnu'  + p^{\prime\nu} \,\hnnu' +q'),
\end{align}
with
\begin{align}
p^{\prime\nu}
= p^{\nu} - g^{\mu\nu} ( u \phi_\mu + \phi_\mu u),\qquad
q'= q - g^{\mu\nu} (\hnmu u \phi_\nu) + g^{\mu\nu} u \phi_\mu \phi_\nu - p^\mu \phi_\mu.
\label{eq-p'q'-pqphi}
\end{align}
\end{proposition}
In this proposition, $\phi_\mu$ is as $p^\mu$: it transforms homogeneously in a change of trivializations of $V$.

\begin{proof}
This is a direct computation using relations like $\hnmu' u = \hnmu u + [\phi_\mu, u]$ and $\hnmu' \phi_\nu = \hnmu \phi_\nu + [\phi_\mu, \phi_\nu]$ in \eqref{eq-P-up'q'-hatnabla} and comparing with \eqref{eq-P-upq-hatnabla}.
\end{proof}

\begin{corollary}
\label{cor-p=0}
There is a unique covariant derivative $\nabla$ such that $p^\mu = 0$. This implies that we can always write $P$ in the reduced form
\begin{align}
\label{eq-P-uq}
P &= - ( g^{\mu\nu}  \hnmu u \hnnu  + q).
\end{align}
\end{corollary}

\begin{proof}
The first part of \eqref{eq-p'q'-pqphi} can be solved in $\phi_\mu$ for the condition $p^{\prime\nu}=0$. Indeed, using results in \cite{Pede76a}, the positivity and invertibility of $u$ implies that for any $\nu$, the equation $u (g^{\mu\nu}\phi_\mu) + (g^{\mu\nu}\phi_\mu) u = p^{\nu}$ has a unique solution given by
\begin{align*}
g^{\mu\nu}\phi_\mu = \tfrac{1}{2} \int_{-\infty}^{+\infty} \frac{u^{it-1/2} \, p^{\nu} \, u^{-it-1/2}}{\cosh(\pi t)} \, dt.
\end{align*}
So, given any covariant derivative to which $p^{\nu}$ is associated as in \eqref{eq-def-P-upq}, we can shift this covariant derivative with the above solution $\phi_\mu$ to impose $p^{\prime\nu}=0$.
\end{proof}

This result extends the one in \cite[Section~1.2.1]{Gilk95a}, which is a key ingredient of the method used there. In the following, we could have started with $P$ written as in \eqref{eq-P-uq}. But, on one hand, we will see that this is not necessary to get $\mathcal{R}_r$ in terms on $u, p^\mu, q$ (at least for $r=2$). On the other hand, we will see in Section~\ref{sec-applications-NCT} that the covariant derivative which is naturally given by the geometric framework of the rational noncommutative torus does not implies $p^\mu=0$, and we will then apply directly the most general result. Obviously, it could be possible to first establish our result for the reduced expression \eqref{eq-P-uq} and then to go to the general result using Prop.~\ref{prop-change-connection-pq}. But this would complicate unnecessarily the presentation of the method and our results.

\subsection{The method}

The method described in \cite{IochMass17a} starts with $P$ written as $P = - g^{\mu\nu} u \pmu \pnu - v^\mu \pmu - w$ and leads to $- P (e^{i x \xi} f) = - e^{i x \xi} \big[ H + K + P ] f$ where $H = g^{\mu\nu} u \xi_\mu \xi_\nu$ and $K = - i \xi_\mu \big( v^\mu +  2 g^{\mu\nu} u \pnu \big)$. 

This can be generalized for a covariant writing of $P$. Using \eqref{eq-def-P-upq}, one gets
\begin{align}
- P (e^{i x \xi} f)
&= \begin{multlined}[t]
e^{i x \xi} \big[ 
-g^{\mu\nu} u \xi_\mu \xi_\nu 
+ i \xi_\mu \big( p^\mu + g^{\mu\nu} (\nnu u) - [\tfrac{1}{2} \alpha^\nu - \beta^\nu ] u + 2 g^{\mu\nu} u \nnu \big) 
\\
g^{\mu\nu} u \nmu \nnu + \big( p^\nu + g^{\mu\nu} (\nmu u) - [\tfrac{1}{2} \alpha^\nu - \beta^\nu] u \big) \nnu + q
\big] f
\end{multlined}
\nonumber
\\
&= - e^{i x \xi} [ H + K + P ] f,
\label{eq-P-exp-HKP-f}
\end{align}
with
\begin{align}
H \vc g^{\mu\nu} u \xi_\mu \xi_\nu, \qquad
K \vc - i \xi_\mu \big( p^\mu + g^{\mu\nu} (\nnu u) - [\tfrac{1}{2} \alpha^\mu - \beta^\mu ] u + 2 g^{\mu\nu} u \nnu \big).
\label{eq-H-K-upq}
\end{align}
These relations look like the expressions of $H$ and $K$ given above (see \cite[eq.~(1.6), (1.7)]{IochMass17a}) with the replacements
\begin{align}
\pmu \mapsto \nmu,
\qquad
v^\mu \mapsto p^\mu + g^{\mu\nu} (\nnu u) - [\tfrac{1}{2} \alpha^\mu - \beta^\mu] u,
\qquad
w \mapsto q.
\label{eq-change-uvw-upq}
\end{align}
As in \cite{IochMass17a}, we have $\Tr[a e^{-tP}\,] = \int \dd x\, \tr[a(x) K(t,x,x)]$ with
\begin{align*}
K(t,x,x) 
&= \tfrac{1}{(2\pi)^d}\int \dd\xi \,e^{-ix.\xi} \,(e^{-tP} \,e^{i x.\xi})
&= \tfrac{1}{(2\pi)^d}\int \dd\xi \,e^{-t(H+K+P)} \,\bbbone 
=\tfrac{1}{t^{d/2}}\tfrac{1}{(2\pi)^d}  \int \dd\xi \,e^{-H-\sqrt{t} K -tP}\,\bbbone. \label{Volterra}
\end{align*} 
Here $\bbbone$ is the constant $1$-valued function. Notice that $K(t,x,x)$ is a density, and that $\abs{g}^{-1/2} K(t,x,x)$ is a true function on $M$. Using the Lebesgue measure $\dd x$ instead of $\dvolg(x)$ is convenient to establish the previous relation which uses Fourier transforms (this point has not been emphasized in \cite{IochMass17a}).

The asymptotics expansion is obtained by the Volterra series 
\begin{align*}
e^{A+B}=e^A +\sum_{k=1}^\infty \int_{\Delta_k} \dd s\, e^{(1-s_1)A}\,B\,e^{(s_1-s_{2})A} \cdots e^{(s_{k-1}-s_k)A} \,B \,e^{s_k A}\,.
\end{align*}
where
\begin{align*}
\Delta_k \vc \{s=(s_1,\dots, s_k)\in \R_+^{k} \, \vert \, 0 \leq s_k \leq s_{k-1} \leq \cdots \leq s_2 \leq s _1 \leq 1\}
\text{ and }\Delta_0 \vc \varnothing \,\,\text{by convention}.
\end{align*}
For $A = -H$ and $B = -\sqrt{t} K -tP$, one gets
\begin{align}
e^{-H-\sqrt{t} K -tP} \bbbone 
= e^{-H }+ \sum_{k=1}^\infty (-1)^k f_k [ (\sqrt{t} K +t P)\otimes \cdots \otimes (\sqrt{t}K +t P) ]
\label{Volt}
\end{align}
with
\begin{align}
f_k(\xi) [ B_1 \otimes \cdots \otimes  B_k ] 
&\vc \int_{\Delta_k} \dd s \, e^{(s_1-1) H(\xi)} \, B_1 \,e^{(s_2-s_1) H(\xi)} \, B_2 \cdots B_k\, e^{-s_k H(\xi)},
\label{eq-fk-def}
\\
f_0(\xi)[\lambda] 
&\vc \lambda\, e^{-H(\xi)},
\nonumber
\end{align}
where $B_i$ are matrix-valued differential operators in $\nmu$ depending on $x$ and (linearly in) $\xi$, and $\lambda\in \C$. Collecting the powers of $\sqrt{t}$, one gets
\begin{align*}
\Tr \,[ a e^{-tP}] 
&\underset{t\downarrow 0}{\simeq} 
t^{- d/2} \sum_{r=0}^\infty a_{r}(a, P) \, t^{r/2}
\end{align*}
Each $a_{r}(a, P)$ contains an integration along $\xi$, which kills all the terms in odd power in $\sqrt{t}$ since $K$ is linear in $\xi$ while $H$ is quadratic in $\xi$: $a_{2n+1}(a, P) = 0$ for any $n\in \N$. For instance, the first two non-zero local coefficients are\footnote{Notice the change with convention in \cite{IochMass17a} : $a_{2r}$ here corresponds to  $a_r$ in \cite{IochMass17a}.}
\begin{align*}
a_0(a, P)(x) 
&= \tfrac{\abs{g}^{-1/2}}{(2 \pi)^{d}} \tr [ a(x) \int \dd\xi\, e^{-H(x,\xi)}  \label{a0(x)} ] ,
\nonumber
\\
a_2(a, P)(x)  
&= \begin{multlined}[t]
\tfrac{\abs{g}^{-1/2}}{(2 \pi)^{d}} \tr \,[ a(x) \int \dd\xi \int_{\Delta_2} \dd s \, e^{(s_1-1)H}\,K\,e^{(s_2 -s_1)H}\, K\,e^{-s_2 H} ]
\\
- \tfrac{1}{(2 \pi)^{d}} \tr \,[ a(x) \int \dd\xi \int_{\Delta_1} \dd s \,e^{(s_1-1)H}\,P\, e^{-s_1 H}]
\end{multlined}
\end{align*}
(remark the coefficient $\abs{g}^{-1/2}$ added here to be compatible with \eqref{eq-araP-dvol}).

The strategy to compute these coefficient is twofold. First, we get rid of the $\nmu$'s in the arguments $B_i$. This is done using \cite[Lemma~2.1]{IochMass17a}, which can be applied here since $\nmu$ is a derivation: by iteration of the relation
\begin{multline}
\label{eq-fk-nabla-propagation}
f_k(\xi)[B_1 \otimes \cdots \otimes B_i \nmu \otimes \cdots \otimes B_k] =\sum_{j=i+1}^k f_k(\xi)[B_1 \otimes \cdots \otimes(\nmu B_j)\otimes \cdots\otimes  B_k]
\\
- \sum_{j=i}^k f_{k+1}(\xi)[B_1 \otimes \cdots \otimes B_j\otimes (\nmu H)\otimes B_{j+1}\otimes \cdots\otimes  B_k], 
\end{multline}
we transform each original term into a sum of operators acting on arguments of the form $B_1 \otimes \cdots \otimes  B_k = \BB_k^{\mu_1\dots \mu_{\ell}}\,\xi_{\mu_1}\cdots \xi_{\mu_{\ell}}$ (for different values of $k$) where now all the $B_i$ are matrix-valued functions (of $x$ and $\xi$), or, equivalently, the $\BB_k^{\mu_1\dots\mu_\ell}$ are $M_N(\C)^{\otimes^k}$-valued functions (of $x$ only).

The second step of the strategy is to compute the operators applied to the arguments $\BB_k^{\mu_1\dots\mu_\ell}$. They all look like
\begin{align*}
\tfrac{1}{(2\pi)^d} \int \dd\xi \, \xi_{\mu_1} \cdots \xi_{\mu_\ell} \, f_k(\xi)[ \BB_k^{\mu_1\dots\mu_\ell} ] \in M_N(\C), \end{align*}
where the $f_k(\xi)$ are defined by \eqref{eq-fk-def} and depend only on $u$ through $H$. As shown in \cite{IochMass17a}, these operators are related to operators $T_{k,p}(x) : M_N(\C)^{\otimes^{k+1}} \to M_N(\C)^{\otimes^{k+1}}$ defined by
\begin{align}
\label{termgeneric}
T_{k,p}(x)
&\vc \tfrac{1}{(2 \pi)^{d}} \int_{\Delta_k} \dd s \int \dd\xi\, \xi_{\mu_1} \cdots \xi_{\mu_{2p}} \, e^{- \norm{\xi}^2 \, C_k(s,u(x))},
\\
T_{0,0}(x) 
&\vc \tfrac{1}{(2 \pi)^{d}} \int \dd\xi\, \, e^{- \norm{\xi}^2 u(x)} \in M_N(\C),
\nonumber
\end{align}
where $\norm{\xi}^2 \vc g^{\mu\nu} \xi_\mu \xi_\nu$ and the $C_k(s, A) : M_N(\C)^{\otimes^{k+1}} \to M_N(\C)^{\otimes^{k+1}}$ are the operators
\begin{align*}
C_k(s,A) [ B_0 \otimes B_1 \otimes \cdots \otimes B_k]
= (1-s_1) \, B_0 A \otimes B_1 \otimes \cdots \otimes B_k
&+ (s_1-s_2)\, B_0 \otimes B_1 A \otimes \cdots \otimes B_k
\\
&+ \cdots 
+ s_k \, B_0 \otimes B_1 \otimes \cdots \otimes B_k  A.
\end{align*}
Denote by $\mul : M_N(\C)^{\otimes^{k+1}} \to M_N(\C), B_0 \otimes B_1 \otimes \cdots \otimes B_k \mapsto  B_0 B_1 \cdots B_k$ the matrix multiplication, then
\begin{align*}
\tfrac{1}{(2\pi)^d} B_0 \int \dd\xi \, \xi_{\mu_1} \cdots \xi_{\mu_\ell} \, f_k(\xi)[ B_1 \otimes \cdots \otimes B_k]
&= \mul\circ T_{k,p}(x)[ B_0 \otimes B_1 \otimes \cdots \otimes B_k],
\end{align*}
so that each function $a_r(a, P)(x)$ is expressed formally as a sum
\begin{align}
\label{eq-a-T-rel}
a_r(a, P)(x)
&= \abs{g}^{-1/2}\sum \tr \big[ \mul\circ T_{k,p}(x)[ a(x) \otimes B_1(x) \otimes \cdots \otimes B_k(x)] \big].
\end{align}
This sum comes form the collection of the original terms in $K$ and $P$ producing the power $t^{r/2}$ and the application of \cite[Lemma~2.1]{IochMass17a} \textit{i.e.} \eqref{eq-fk-nabla-propagation}. This sum relates the $r$ on the LHS to the possible couples $(k,p)$ on the RHS. The $B_i$ are matrix-valued functions (of $x$) expressed in terms of the original constituents of $H$, $K$, and $P$ and their covariant derivatives.

Let us mention here how the procedure introduced in \cite{IochMass17a} is adapted to the situation where we have the left factor $a(x)$: in \cite{IochMass17a}, the relation between the $T_{k,p}(x)$ and the $f_k(\xi)$ used a trick which consist to add a $B_0 = \bbbone$ argument in front of $B_1 \otimes \cdots \otimes B_k$ (the purpose of the $\kappa$ map defined in \cite{IochMass17a}). Here, $\bbbone$ is simply replaced by $a(x)$. But, since 
\begin{equation*}
\mul\circ T_{k,p}(x)[ B_0 \otimes B_1(x) \otimes \cdots \otimes B_k(x)] = B_0\, \mul\circ T_{k,p}(x)[ \bbbone \otimes B_1(x) \otimes \cdots \otimes B_k(x)], 
\end{equation*}
it is now easy to propose an expression for the factor $\mathcal{R}_r$ as a sum
\begin{align}
\label{eq-R-Tkp}
\mathcal{R}_r = \abs{g}^{-1/2}\sum \mul\circ T_{k,p}(x)[ \bbbone \otimes B_1(x) \otimes \cdots \otimes B_k(x)].
\end{align}

One of the main result of \cite{IochMass17a} is to express the operators $T_{k,p}$ in terms of universal functions through a functional calculus relation involving the spectrum of $u$ (these relations take place at any fixed value of $x \in M$, that we omit from now on). For $r_i > 0$, $\alpha\in \R$, and $k \in \N$, let 
\begin{align*}
I_{\alpha,k}(r_0, r_1, \dots, r_k) 
&\vc 
\int_{\Delta_k} \dd s\, [(1-s_1)r_0 + (s_1-s_2) r_1 + \dots + s_k r_k]^{-\alpha} 
\\
&=
\int_{\Delta_k} \dd s\, [r_0 + s_1 (r_1 - r_0) + \dots + s_k (r_k - r_{k-1})]^{-\alpha},
\\
\shortintertext{so that}
I_{\alpha,k}(r_0, \dots, r_0) & =
\tfrac{1}{k!} r_0^{-\alpha}.
\end{align*}
In these functions, the arguments $r_i > 0$ are in the spectrum of the positive matrix $u$.

Denote by $R_i(A) : M_N(\C)^{\otimes^{k+1}} \to M_N(\C)^{\otimes^{k+1}}$ the right multiplication on the $i$-th factor
\begin{align*}
R_i(A) [ B_0 \otimes B_1 \otimes \cdots \otimes B_k]
\vc B_0 \otimes B_1 \otimes \cdots \otimes B_i A \otimes \cdots \otimes B_k,
\end{align*}
then 
\begin{align*}
T_{k,p}&
= g_d \, G(g)_{\mu_1\dots \mu_{2p}} \, I_{d/2+p,k} \big(R_0(u),R_1(u),\dots,R_k(u)\big),
\end{align*}
with
\begin{align}
g_d 
&\vc \tfrac{1}{(2\pi)^d} \int_{\R^d} \dd\xi\, e^{-\abs{\xi}_{g(x)}^2}
= \tfrac{\abs{g}^{1/2}}{2^{d}\,\pi^{d/2}}\,, 
\label{eq-gd}
\\
G(g)_{\mu_1\dots \mu_{2p}} 
&\vc  \tfrac{1}{(2\pi)^d\,g_d} \int \dd\xi\, \xi_{\mu_1} \cdots \xi_{\mu_{2p}}\, e^{-g^{\alpha\beta} \xi_\alpha\xi_\beta}
\nonumber
\\
&\, = \tfrac{1}{2^{2p}\,p!} \, ( \sum_{\rho \in S_{2p}} g_{\mu_{\rho(1)} \mu_{\rho(2)}}\cdots g_{\mu_{\rho(2p-1)} \mu_{\rho(2p)}} )
= \tfrac{(2p)!}{2^{2p}\,p!} \,g_{(\mu_1\mu_2}\cdots g_{\mu_{2p-1}\mu_{2p})},
\nonumber 
\end{align}
where $S_{2p}$ is the symmetric group of permutations on $2p$ elements and the parenthesis in the index of $g$ is the complete symmetrization over all indices. Notice that the factor $\abs{g}^{1/2}$ in $g_d$ simplifies with the factor $\abs{g}^{-1/2}$ in \eqref{eq-R-Tkp}.

The universal functions $I_{\alpha,k}$ have been studied in \cite[Section~3]{IochMass17a}. They satisfy a recursive formula valid for $1\neq \alpha\in \R$ and $k\in \N^*$:
\begin{align}
\label{eq-I-recursive}
I_{\alpha,k}(r_0,\dots,r_k) = \tfrac{1}{(\alpha-1)}(r_{k-1}-r_k)^{-1}[I_{\alpha-1,k-1}(r_0,\dots,r_{k-2},r_k)-I_{\alpha-1,k-1} (r_0,\dots,r_{k-1})].
\end{align}
It is possible to give some expressions for the $I_{\alpha,k}$ for any $(\alpha,k)$. They depend on the parity of $d$. For $d$ even, the main results are that $I_{n,k}$ are Laurent polynomials for $\N \ni n = (d - r)/2 + k \geq k+1$ ($d \geq r + 2$) and $k\in \N^*$, while they exhibit a more complicated expression in terms of $\log$ functions for $\N \ni n = (d - r)/2 + k \leq k$ ($d \leq r$). For $d$ odd, the $I_{n,k}$ can be expressed in terms of square roots of the $r_i$, but without an \textit{a priori} general expression.

The recursive formula \eqref{eq-I-recursive} can be used to write any $I_{\alpha,k}$ appearing in the computation of the operators $T_{k,p}$ in terms of $I_{\alpha-k+1,1}$. The case $\alpha=1$ appears in dimension $d=2$: the fundamental spectral function is $I_{1,1}$, and a direct computation shows that
\begin{equation*}
I_{1,1} = \frac{\ln(r_0) - \ln(r_1)}{r_0 - r_1}
\end{equation*}
Using $\frac{x}{e^x - 1} = \sum_{n=0}^\infty \tfrac{B_n}{n!} x^n$, where $B_n$ are the Bernoulli numbers, one gets, with $x= \ln(r_0) - \ln(r_1)$,
\begin{equation*}
r_1\, I_{1,1}(r_0, r_1) = \sum_{n=0}^\infty \tfrac{B_n}{n!} \,[ \ln(r_0) - \ln(r_1) ]^n.
\end{equation*}
A relation between the Bernoulli numbers and $a_2(a, P)$ has already been noticed in the computation of the modular curvature for the noncommutative two torus in \cite{ConnMosc14a} (see Section~\ref{sec-applications-NCT}).

\subsection{\texorpdfstring{The results for $a_2(a,P)$}{The results for a2(a,P)}}
\label{subsec-results}

In the following, we restrict ourselves to the computation of $a_2(a,P)$. This section gives the main results of the paper. The computations are detailed in Section~\ref{sec-details-computations}.

Let us introduce the following notation. For any $x \in M$, denote by $r_i = r_i(x) > 0$ an element in the (discrete) spectrum $\spec(u)$ of $u=u(x)$ and by $E_{r_i} = E_{r_i}(x)$ the associated projection of $u$. This implies that
\begin{equation*}
u = \sum_{r_0 \in \spec(u)} r_0 E_{r_0} = r_0 E_{r_0}
\end{equation*}
where in the last expression we omit the summation over $r_0$, as will be the case in many expressions given in the following. Notice that $\bbbone = \sum_{r_0 \in \spec(u)} E_{r_0}$ and $E_{r_0} E_{r_1} = \delta_{r_0, r_1} E_{r_0}$.

\begin{theorem}
\label{thm-R-uvw}
For $P$ given by \eqref{eq-def-P-uvw}, $a_2(a,P)(x) = \tr[ a(x) \mathcal{R}_2(x) ]$ with
\begin{align}
\mathcal{R}_2
&= \tfrac{1}{2^{d}\,\pi^{d/2}} \big[ 
\begin{aligned}[t]
&\alpha\, r_0^{-d/2+1}\, 
E_{r_0}
+ F_{\partial u}^\mu(r_0, r_1) \, 
E_{r_0} (\pmu u) E_{r_1}
\\[1mm]
& + g^{\mu\nu} F_{\partial \partial u}(r_0, r_1) \, 
E_{r_0} (\pmu \pnu u) E_{r_1}
+ g^{\mu\nu} F_{\partial u, \partial u}(r_0, r_1, r_2) \, 
E_{r_0} (\pmu u) E_{r_1} (\pnu u) E_{r_2}
\\[1mm]
& + F_{w}(r_0, r_1) \, 
E_{r_0} w E_{r_1}
+ F_{v, \mu}(r_0, r_1) \, 
E_{r_0} v^\mu E_{r_1}
\\[1mm]
& + F_{v,\partial u}(r_0, r_1, r_2) \, 
E_{r_0} v^\mu E_{r_1} (\pmu u) E_{r_2}
+ F_{\partial u,v}(r_0, r_1, r_2) \, 
E_{r_0} (\pmu u) E_{r_1} v^\mu E_{r_2}
\\[1mm]
& + g_{\mu\nu} F_{v,v}(r_0, r_1, r_2) \, 
E_{r_0} v^\mu E_{r_1} v^\nu E_{r_2}
+ F_{\partial v}(r_0, r_1) \, 
E_{r_0} (\pmu v^\mu) E_{r_1}
 \big],
 \end{aligned}
\label{eq-R-uvw-functions}
\end{align}
where the sums over the $r_0, r_1, r_2$ in the spectrum of $u$ are omitted, the spectral functions $F$ are given below, and
\begin{align*}
\alpha
&\vc \begin{aligned}[t]
& \tfrac{1}{3} (\pmu\pnu g^{\mu\nu}) 
-\tfrac{1}{12} g^{\mu\nu} g_{\rho \sigma}(\pmu\pnu g^{\rho\sigma}) 
+ \tfrac{1}{48} g^{\mu\nu} g_{\rho \sigma} g_{\alpha\beta}(\pmu g^{\rho\sigma})(\pnu g^{\alpha\beta})
\\[1mm]
& + \tfrac{1}{24} g^{\mu\nu} g_{\rho \sigma} g_{\alpha\beta}(\pmu g^{\rho\alpha})(\pnu g^{\sigma\beta}) 
-\tfrac{1}{12} g_{\rho\sigma}(\pmu g^{\mu\nu})(\pnu g^{\rho\sigma}) 
\\[1mm]
& + \tfrac{1}{12} g_{\rho\sigma}(\pmu g^{\nu\rho})(\pnu g^{\mu\sigma}) 
- \tfrac{1}{4} g_{\rho\sigma}(\pmu g^{\mu\rho})(\pnu g^{\nu\sigma}).
\end{aligned}
\end{align*}
\end{theorem}

\bigskip
The spectral functions in \eqref{eq-R-uvw-functions} are given in terms of the universal function $I_{d/2, 1}$ by
\begin{align*}
F_{w}(r_0, r_1)
&= I_{d/2,1}(r_0, r_1),
\\[1mm]
F_{\partial v}(r_0, r_1)
&= 2 r_0 \frac{I_{d/2,1}(r_0, r_0) - I_{d/2,1}(r_0, r_1)}{d (r_0 - r_1)} ,
\\[1mm]
F_{\partial \partial u}(r_0, r_1)
&= - r_0 \frac{4 r_0 I_{d/2,1}(r_0, r_0) + \big( (d - 4) r_0 - d r_1 \big) I_{d/2,1}(r_0, r_1)}%
{d (r_0 -  r_1)^2} ,
\\[1mm]
F_{\partial u}^\mu(r_0, r_1) 
&= [ \tfrac{1}{2} \alpha^\mu - \beta^\mu]
r_0 \frac{4 r_0 I_{d/2,1}(r_0, r_0) + \big( (d - 4) r_0 - d r_1 \big) I_{d/2,1}(r_0, r_1)}%
{d (r_0 -  r_1)^2} ,
\\[1mm]
F_{v, \mu}(r_0, r_1)
&= - \alpha_\mu 
r_0 \frac{I_{d/2,1}(r_0, r_0) - I_{d/2,1}(r_0, r_1)}{d (r_0 - r_1)} 
- \tfrac{1}{2} [ \tfrac{1}{2} \alpha_\mu - \beta_\mu ] I_{d/2,1}(r_0, r_1),
\\[1mm]
F_{v,v}(r_0, r_1, r_2)
&= \frac{ I_{d/2,1}(r_0, r_1) - I_{d/2,1}(r_0, r_2) }{d (r_1 - r_2)} ,
\\[1mm]
F_{\partial u,v}(r_0, r_1, r_2)
&= \frac{2 r_0}{d} [ 
\frac{I_{d/2,1}(r_0, r_0)}{(r_0 - r_1)(r_0 - r_2)}
+ \frac{I_{d/2,1}(r_0, r_1)}{(r_1 - r_0)(r_1 - r_2)}
+ \frac{I_{d/2,1}(r_0, r_2)}{(r_2 - r_0)(r_2 - r_1)}
],
\\[1mm]
F_{v,\partial u}(r_0, r_1, r_2)
&= \begin{multlined}[t]
- 2 r_0 \frac{I_{d/2,1}(r_0, r_0)}{d (r_0 - r_2) (r_1 - r_2)} 
- 2 r_1 \frac{I_{d/2,1}(r_0, r_1)}{d (r_1 - r_2)^2}  
\\
- \frac{\big(  (d - 4) r_0 r_1 - (d - 2) r_0 r_2 - (d - 2) r_1 r_2 + d r_2^2  \big) I_{d/2,1}(r_0, r_2)}%
{d (r_0 - r_2) (r_1 - r_2)^2}  ,
\end{multlined}
\\[1mm]
F_{\partial u, \partial u}(r_0, r_1, r_2) 
&= \begin{multlined}[t]
\frac{4 r_0}{d (r_0 - r_1) (r_0 - r_2)^2 (r_1 - r_2)^2} 
\\
\times \big[
r_0 (r_1 - r_2) (r_0 - 2 r_1 + r_2) I_{d/2,1}(r_0, r_0)
+ r_1 (r_0 - r_2)^2 I_{d/2,1}(r_0, r_1)
\\
+ \tfrac{1}{2}(r_0 - r_1) \big( (d - 4) r_0 r_1 - (d - 2) r_0 r_2 - d r_1 r_2 + (d + 2) r_2^2 \big) I_{d/2,1}(r_0, r_2)
\big].
\end{multlined}
\end{align*}

\begin{theorem}
\label{thm-R-upq}
For $P$ given by \eqref{eq-def-P-upq}, $a_2(a,P)(x) = \tr[ a(x) \mathcal{R}_2(x) ]$ with
\begin{align}
\mathcal{R}_2
&= \tfrac{1}{2^{d}\,\pi^{d/2}}\, [\,
\begin{aligned}[t]
& \tfrac{1}{6} R \, r_0^{-d/2+1} \, 
E_{r_0}
 + G_{q}(r_0, r_1) \, 
E_{r_0} q E_{r_1}
+ g^{\mu\nu} G_{\hnabla \hnabla u}(r_0, r_1) \, 
E_{r_0} (\hnmu \hnnu u) E_{r_1}
\\[1mm]
& + G_{\hnabla p}(r_0, r_1) \, 
E_{r_0} (\hnmu p^\mu) E_{r_1}
+ g^{\mu\nu}  G_{\hnabla u, \hnabla u}(r_0, r_1, r_2) \, 
E_{r_0} (\hnmu u) E_{r_1} (\hnnu u) E_{r_2}
\\[1mm]
& + G_{p, \hnabla u}(r_0, r_1, r_2) \, 
E_{r_0} p^\mu E_{r_1} (\hnmu u) E_{r_2}
+ G_{\hnabla u, p}(r_0, r_1, r_2) \, 
E_{r_0} (\hnmu u) E_{r_1} p^\mu E_{r_2}
\\[1mm]
& + G_{p, p}(r_0, r_1, r_2) \, 
E_{r_0} p^\mu E_{r_1} p_\mu E_{r_2}
\,]
\end{aligned}
\label{eq-R-upq-functions}
\end{align}
where the sums over the $r_0, r_1, r_2$ in the spectrum of $u$ are omitted, the spectral functions $G$ are given below, and $R$ is the scalar curvature of $g$.
\end{theorem}

The spectral functions in \eqref{eq-R-upq-functions} are given in terms of the spectral functions $F$ by
\begin{align*}
G_q(r_0, r_1)
&= F_{w}(r_0, r_1),
\\[1mm]
G_{\hnabla \hnabla u}(r_0, r_1)
&= F_{\partial \partial u}(r_0, r_1) + F_{\partial v}(r_0, r_1),
\\[1mm]
G_{\hnabla p}(r_0, r_1)
&= F_{\partial v}(r_0, r_1),
\\[1mm]
G_{\hnabla u, \hnabla u}(r_0, r_1, r_2)
&= F_{\partial u, \partial u}(r_0, r_1, r_2)
+ F_{v,\partial u}(r_0, r_1, r_2)
+ F_{\partial u,v}(r_0, r_1, r_2)
+ F_{v,v}(r_0, r_1, r_2)
\\[1mm]
G_{p, \hnabla u}(r_0, r_1, r_2)
&= F_{v,\partial u}(r_0, r_1, r_2)
+ F_{v,v}(r_0, r_1, r_2),
\\[1mm]
G_{\hnabla u, p}(r_0, r_1, r_2)
&= F_{\partial u,v}(r_0, r_1, r_2)
+ F_{v,v}(r_0, r_1, r_2),
\\[1mm]
G_{p, p}(r_0, r_1, r_2)
&= F_{v,v}(r_0, r_1, r_2).
\end{align*}

As shown in \cite{IochMass17a}, the universal spectral functions $I_{\alpha,k}$ are continuous, so that all the spectral functions $F$ and $G$ are also continuous, as can be deduced from their original expressions in terms of functions $I_{\alpha,k}$ given in the list \eqref{eq-def-F-I} and the above relations between the $F$ and the $G$.

\begin{remark}
\label{rem-metrics}
The metric $g$ plays a double role here: it is the metric of the Riemannian manifold $(M, g)$ and it is the non-degenerate tensor which multiply $u$ in $P$. If one has to consider two operators $P_1$ and $P_2$ with tensors $g_1$ and $g_2$ on the same manifold $M$, it may be not natural to take $g_1$ or $g_2$ as the Riemannian metric on $M$. It is possible to consider a metric $h$ on $M$ different to the tensor $g$ associated to $P$. In that case, we have to replace $\dvolg$ in \eqref{eq-araP-dvol} by $\dvolh$ and $K(t,x,x,)$ is then a density for $h$, so that the true function is now $\abs{h}^{-1/2} K(t,x,x)$, and $\abs{h}^{-1/2}$ appears in \eqref{eq-a-T-rel} and \eqref{eq-R-Tkp} in place of $\abs{g}^{-1/2}$. Now, the computation of $T_{k,p}$ makes apparent the coefficient $g_d$ given by \eqref{eq-gd} where the metric $g$ comes from $P$. Finally, in \eqref{eq-R-uvw-functions} and \eqref{eq-R-upq-functions}, the two determinants do not simplify anymore, and one gets an extra factor $\abs{g}^{1/2} \abs{h}^{-1/2}$ in front of $\mathcal{R}_2$, which is now relative to $\varphi_h(a) \vc \int_M  \tr[a(x)] \, \dvolh(x)$.
\end{remark}

A change of connection as in Prop.~\ref{prop-change-connection-pq} does not change the value of $\mathcal{R}_2$. This induces the following relations between the spectral functions $G$.
\begin{proposition}
\label{prop-relations-G}
The spectral functions $G$ satisfy the relations:
\begin{align*}
G_{\hnabla p}(r_0, r_1)
&=
- \frac{r_0 G_{q}(r_0, r_1) + (r_0 - r_1) G_{\hnabla \hnabla u}(r_0, r_1)}{r_0 + r_1},
\\[1mm]
G_{\hnabla u, p}(r_0, r_1, r_2)
&=
- \frac{r_2 G_{q}(r_0, r_2) + (r_0 + 3 r_2) G_{\hnabla \hnabla u}(r_0, r_2) + (r_0 + r_2) (r_1 - r_2) G_{\hnabla u, \hnabla u}(r_0, r_1, r_2)}{(r_0 + r_2) (r_1 + r_2)},
\\[1mm]
G_{p, \hnabla u}(r_0, r_1, r_2)
&=
\frac{r_0 G_{q}(r_0, r_2) + (3 r_0 + r_2) G_{\hnabla \hnabla u}(r_0, r_2) + (r_0 + r_2) (r_1 - r_0) G_{\hnabla u, \hnabla u}(r_0, r_1, r_2)}{(r_0 + r_2) (r_1 + r_0)},
\\[1mm]
G_{p, p}(r_0, r_1, r_2)
&=
- \frac{r_1 G_{q}(r_0, r_2) - (r_0 - 2 r_1 + r_2) G_{\hnabla \hnabla u}(r_0, r_2) + (r_0 - r_1) (r_1 - r_2) G_{\hnabla u, \hnabla u}(r_0, r_1, r_2)}{(r_0 + r_1) (r_1 + r_2)}.
\end{align*}
\end{proposition}

\begin{proof}
Inserting the relations \eqref{eq-p'q'-pqphi} into \eqref{eq-R-upq-functions}, all the terms involving $\phi_\mu$ must vanish. This induces the following relations between the $G$ functions: 
\begin{align*}
r_0 G_{q}(r_0, r_1) + (r_0 - r_1) G_{\hnabla \hnabla u}(r_0, r_1) + (r_0 + r_1) G_{\hnabla p}(r_0, r_1)
&=0,
\\[1mm]
G_{\hnabla p}(r_0, r_2) - (r_0 - r_1) G_{\hnabla u, p}(r_0, r_1, r_2) - (r_0 + r_1) G_{p, p}(r_0, r_1, r_2)
&=0,
\\[1mm]
G_{q}(r_0, r_2) + G_{\hnabla p}(r_0, r_2) + (r_1 - r_2) G_{p, \hnabla u}(r_0, r_1, r_2) + (r_1 + r_2) G_{p, p}(r_0, r_1, r_2)
&=0,
\\[1mm]
2 G_{\hnabla \hnabla u}(r_0, r_2) - G_{\hnabla p}(r_0, r_2)
- (r_0 - r_1) G_{\hnabla u, \hnabla u}(r_0, r_1, r_2) - (r_0 + r_1) G_{p, \hnabla u}(r_0, r_1, r_2)
&=0,
\\[1mm]
\begin{multlined}[b]
G_{q}(r_0, r_2) + 2 G_{\hnabla \hnabla u}(r_0, r_2) + G_{\hnabla p}(r_0, r_2)
\\
+ (r_1 - r_2) G_{\hnabla u, \hnabla u}(r_0, r_1, r_2) + (r_1 + r_2) G_{\hnabla u, p}(r_0, r_1, r_2)
\end{multlined}
&=0,
\\[1mm]
\begin{multlined}[b]
r_0 G_{q}(r_0, r_2) + (r_0 - 2 r_1 + r_2) G_{\hnabla \hnabla u}(r_0, r_2) + (r_0 - r_2) G_{\hnabla p}(r_0, r_2) 
\\
+ (r_0 - r_1)(r_1 - r_2) G_{\hnabla u, \hnabla u}(r_0, r_1, r_2) + (r_0 + r_1)(r_1 - r_2) G_{p, \hnabla u}(r_0, r_1, r_2) 
\\
+ (r_0 - r_1)(r_1 + r_2) G_{\hnabla u, p}(r_0, r_1, r_2) + (r_0 + r_1)(r_1 + r_2) G_{p, p}(r_0, r_1, r_2)
\end{multlined}
&=0.
\end{align*}
One can check directly that these relations hold true. From them, one can solve $G_{\hnabla p}, G_{\hnabla u, p}, G_{p, \hnabla u}$, and $G_{p, p}$ in terms of $G_{q}, G_{\hnabla \hnabla u}$, and $G_{\hnabla u, \hnabla u}$. This gives the relations of the proposition.
\end{proof}

These relations show that the four spectral functions $G$ involved in terms with $p^\mu$ are deduced from the three spectral functions involving only $u$ and $q$. This result is not a surprise: from Corollary~\ref{cor-p=0} we know that we can start with $p^\mu = 0$, so that $\mathcal{R}_2$ is written in terms of the three functions $G_{q}, G_{\hnabla \hnabla u}, G_{\hnabla u, \hnabla u}$ only, and then we can change the connection in order to produce the most general expression for $\mathcal{R}_2$. In other words, among the seven spectral functions $G$, only three are fundamental.

\medskip
The spectral functions $G$ can be computed explicitly, and their expressions depends on the value of $m$. Let
\begin{align*}
Q_1(a, b, c)
&\vc
\frac{-3 a^3 + a^2 b - 6 a^2 c + 6 a b c + a c^2 + b c^2}{2 (a-b)^2 (a-c)^3},
\\
Q_2(a,b,c)
&\vc \frac{1}{2}\, [ \begin{multlined}[t]
\frac{2}{(a - b) (a - c)}
+ \frac{(a + b)(a + c) - 4 a^2}{(a - b)^2 (a - c)^2} \ln(a)
\\[1.5mm]
\hspace{0.3cm}+ \frac{a + b}{(b - a)^2 (b - c)} \ln(b)
+ \frac{a + c}{(c - a)^2 (c - b)} \ln(c)
\,],
\end{multlined}
\\
Q_3(a,b,c)
&\vc \begin{multlined}[t]
6 \sqrt{a b c } + a^{3/2} + 2 b^{3/2} + c^{3/2} + \sqrt{a c}\, ( \sqrt{a} + \sqrt{c}) 
\\
\hspace{0.8cm}+ 2 \big(
a (\sqrt{b} + \sqrt{c}) + 2 b (\sqrt{a} + \sqrt{c}) + c (\sqrt{a} + \sqrt{b})
\big),
\end{multlined}
\\
Q_4(a,b,c)
&\vc
\frac{a + b + c + 2 \sqrt{a b} + 2 \sqrt{a c} + \sqrt{b c}}{\sqrt{b c} (\sqrt{a} + \sqrt{b})^2 (\sqrt{a} + \sqrt{c})^2 (\sqrt{b} + \sqrt{c})}.
\end{align*}

\begin{corollary}[Case $d = 2$]
\label{cor-R-d=2}
In dimension two the spectral functions $G$ can be written in terms of log functions, and this leads to:
\begin{align*}
\mathcal{R}_2 
&=  \tfrac{1}{4\,\pi} \big[ 
\begin{aligned}[t]
&\tfrac{1}{6} R 
+ \frac{\ln(r_0) - \ln(r_1)}{r_0 -  r_1}
E_{r_0} q E_{r_1}
+ \frac{1}{r_0 -  r_1} ( 
1 - r_0 \frac{\ln(r_0) - \ln(r_1)}{r_0 -  r_1}
)
E_{r_0} (\hnabla_\mu p^\mu) E_{r_1}
\\[2mm]
& - \frac{1}{\left(r_0 - r_1\right)^2} ( 
r_0 + r_1 - 2 r_0 r_1 \frac{\ln(r_0) - \ln(r_1)}{r_0 -  r_1}
)
g^{\mu\nu} 
E_{r_0} (\hnabla_\mu \hnabla_\nu u) E_{r_1}
\\[2mm]
& + \big( \begin{aligned}[t]
& \frac{\left(r_0 + r_2\right) \left(r_0 - 2 r_1 + r_2\right)}{\left(r_0 - r_1\right) \left(r_0 - r_2\right)^2 \left(r_1 - r_2\right)}
- Q_1(r_0, r_1, r_2) \ln(r_0)
\\[1.5mm]
& - \frac{\left(r_0 + r_1\right) \left(r_1 + r_2\right)}{2 \left(r_0 - r_1\right)^2 \left(r_1 - r_2\right)^2}  \ln(r_1)
- Q_1(r_2, r_1, r_0) \ln(r_2)
\big)
g^{\mu\nu} 
E_{r_0} (\hnabla_\mu u) E_{r_1} (\hnabla_\nu u) E_{r_2}
\end{aligned}
\\[2mm]
& + 
Q_2(r_0, r_1, r_2) \,
E_{r_0} (\hnabla_\mu u) E_{r_1} p^\mu E_{r_2}
- Q_2(r_2, r_1, r_0) \,
E_{r_0} p^\mu E_{r_1} (\hnabla_\mu u) E_{r_2}
\\[2mm]
& + \frac{1}{2} (
\frac{\ln(r_0)}{(r_0 - r_1) (r_0 - r_2)} 
+ \frac{\ln(r_1)}{(r_1 - r_0) (r_1 - r_2)} 
+ \frac{\ln(r_2)}{(r_2 - r_0) (r_2 - r_1)} 
)
E_{r_0} p_\mu E_{r_1} p^\mu E_{r_2}
\big].
\end{aligned}
\end{align*}
\end{corollary}

\bigskip
When $d =2m  \geq 4$ is even, all the functions involved are Laurent polynomials. This is a direct consequence of \cite[Prop.~3.5]{IochMass17a}:
\begin{align*}
G_{q}(r_0, r_1)
&=
\sum_{0 \leq \ell \leq m-2}  \tfrac{1}{m-1} \, r_0^{\ell+1-m} r_1^{-\ell-1},
\\[1mm]
G_{\hnabla \hnabla u}(r_0, r_1)
&=
- \sum_{0 \leq \ell \leq m-2} \tfrac{(m-\ell-1)(\ell+1)}{m(m-1)} \, r_0^{\ell+1-m} r_1^{-\ell-1} ,
\\[1mm]
G_{\hnabla p}(r_0, r_1)
&=
- \sum_{0 \leq \ell \leq m-2} \tfrac{m-\ell-1}{m(m-1)} \, r_0^{\ell+1-m} r_1^{-\ell-1},
\\[1mm]
G_{\hnabla u, \hnabla u}(r_0, r_1, r_2)
&=
- \sum_{0 \leq \ell \leq k \leq m-2} \tfrac{(2\ell+1)(2k-2m+3)}{2m(m-1)} \, r_0^{k+1-m} r_1^{\ell-k-1} r_2^{-\ell-1} ,
\\[1mm]
G_{p, \hnabla u}(r_0, r_1, r_2)
&=
- \sum_{0 \leq \ell \leq k \leq m-2} \tfrac{2\ell+1}{2m(m-1)} \, r_0^{k+1-m} r_1^{\ell-k-1} r_2^{-\ell-1},
\\[1mm]
G_{\hnabla u, p}(r_0, r_1, r_2)
&=
- \sum_{0 \leq \ell \leq k \leq m-2} \tfrac{2k-2m+3}{2m(m-1)} \, r_0^{k+1-m} r_1^{\ell-k-1} r_2^{-\ell-1},
\\[1mm]
G_{p, p}(r_0, r_1, r_2)
&=
- \sum_{0 \leq \ell \leq k \leq m-2} \tfrac{1}{2m(m-1)} \, r_0^{k+1-m} r_1^{\ell-k-1} r_2^{-\ell-1}.
\end{align*}

This implies the following expressions for $\mathcal{R}_2$:
\begin{corollary}[Case $d = 2m$ even and $d \geq 4$]
\label{cor-R-d>2-even}
Using the expressions of the spectral functions $G$ as Laurent polynomials, one has
\begin{align*}
\mathcal{R}_2 
&= \tfrac{1}{2^{2m}\pi^{m}} \big[
\begin{aligned}[t]
& \tfrac{1}{6} R \, u^{-m+1}
 + \sum_{0 \leq \ell \leq m-2}  \tfrac{1}{m-1} \, u^{\ell+1-m} q u^{-\ell-1}
 - \sum_{0 \leq \ell \leq m-2} \tfrac{m-\ell-1}{m(m-1)} \, u^{\ell+1-m} (\hnabla_\mu p^\mu) u^{-\ell-1} 
\\[1mm]
& - \sum_{0 \leq \ell \leq m-2} \tfrac{(m-\ell-1)(\ell+1)}{m(m-1)} \, g^{\mu\nu} u^{\ell+1-m} (\hnabla_\mu \hnabla_\nu u) u^{-\ell-1} 
\\[1mm]
& - \sum_{0 \leq \ell \leq k \leq m-2} \tfrac{(2\ell+1)(2k-2m+3)}{2m(m-1)} \, g^{\mu\nu} u^{k+1-m} (\hnabla_\mu u) u^{\ell-k-1} (\hnabla_\nu u) u^{-\ell-1}  
\\[1mm]
& - \sum_{0 \leq \ell \leq k \leq m-2} \tfrac{2k-2m+3}{2m(m-1)} \, u^{k+1-m} (\hnabla_\mu u) u^{\ell-k-1} p^\mu u^{-\ell-1}
\\[1mm]
& - \sum_{0 \leq \ell \leq k \leq m-2} \tfrac{2\ell+1}{2m(m-1)} \, u^{k+1-m} p^\mu u^{\ell-k-1} (\hnabla_\mu u) u^{-\ell-1}
\\[1mm]
& - \sum_{0 \leq \ell \leq k \leq m-2} \tfrac{1}{2m(m-1)} \, u^{k+1-m} p_\mu u^{\ell-k-1} p^\mu u^{-\ell-1} 
\big]
\end{aligned}
\\[1mm]
&= \tfrac{1}{2^{2m}\pi^{m}} 
\begin{aligned}[t]
&\big[ \tfrac{1}{6} R \, u^{-m+1} 
 + \sum_{0 \leq \ell \leq m-2}  \tfrac{1}{m-1} \, u^{\ell+1-m} q u^{-\ell-1}
\\[1mm]
& 
- \!\!\!\! \sum_{0 \leq \ell \leq m-2} \tfrac{(m-\ell-1)}{m(m-1)} \, g^{\mu\nu}
u^{\ell+1-m} \big[ \hnabla_\mu \left( (\ell+1) \hnabla_\nu u + p_\nu \right) \big] u^{-\ell-1}
\\[1mm]
& - \!\!\!\!\!\! \sum_{0 \leq \ell \leq k \leq m-2} \!\!\!\! \tfrac{1}{2m(m-1)}  \,   
 g^{\mu\nu} u^{k+1-m} \left[ \left( 2k-2m+3 \right) \! \hnabla_\mu u + p_\mu \right] u^{\ell-k-1} \left[ \left( 2\ell+1 \right) \! \hnabla_\nu u + p_\nu \right] u^{-\ell-1}
\big].
\end{aligned}
\end{align*}
\end{corollary}

\bigskip
\begin{corollary}[Case $d=4$]
\label{cor-d=4}
In dimension four these expressions simplify further to:
\begin{align*}
\mathcal{R}_2
&= \tfrac{1}{16\pi^{2}} [
\begin{aligned}[t]
& \tfrac{1}{6} R \, u^{-1}
 + u^{-1} q u^{-1}
 - \tfrac{1}{2} \, u^{-1} (\hnabla_\mu p^\mu) u^{-1} 
- \tfrac{1}{2} \, g^{\mu\nu} u^{-1} (\hnabla_\mu \hnabla_\nu u) u^{-1} 
\\[1mm]
& + \tfrac{1}{4} \, g^{\mu\nu} u^{-1} (\hnabla_\mu u) u^{-1} (\hnabla_\nu u) u^{-1}  
+ \tfrac{1}{4} \, u^{-1} (\hnabla_\mu u) u^{-1} p^\mu u^{-1}
- \tfrac{1}{4} \, u^{-1} p^\mu u^{-1} (\hnabla_\mu u) u^{-1}
\\[1mm]
& - \tfrac{1}{4} \, u^{-1} p_\mu u^{-1} p^\mu u^{-1} 
]
\end{aligned}
\\[1mm]
&= \tfrac{1}{16\pi^{2}} \big[
\begin{aligned}[t]
& \tfrac{1}{6} R \, u^{-1} 
 \!+ \! u^{-1} q u^{-1}
\!-\! \tfrac{1}{2} \, g^{\mu\nu}
u^{-1} [ \hnabla_\mu ( \hnabla_\nu u + p_\nu ) ] u^{-1}
\!+\! \tfrac{1}{4}  \,   
 g^{\mu\nu} u^{-1} [ \hnabla_\mu u - p_\mu ] u^{-1} [ \hnabla_\nu u + p_\nu ] u^{-1}
\big].
\end{aligned}
\end{align*}
\end{corollary}

\begin{corollary}[Case $d=3$]
\label{cor-R-d=3}
In dimension three the spectral functions $G$ can be written in terms of square roots of the $r_i$ (see \cite[Prop.~3.4]{IochMass17a}), and this leads to:
\begin{align*}
\mathcal{R}_2 
&= \tfrac{1}{8\pi^{3/2}} \big[
\begin{aligned}[t]
&\tfrac{1}{6} R 
+ \frac{2}{\sqrt{r_0 r_1} (\sqrt{r_0} + \sqrt{r_1})}
E_{r_0} q E_{r_1}
- \tfrac{2}{3} \frac{2 \sqrt{r_0} + \sqrt{r_1}}{\sqrt{r_0 r_1} (\sqrt{r_0} + \sqrt{r_1})^2} \,
E_{r_0} (\hnabla_\mu p^\mu) E_{r_1}
\\[2mm]
& - \tfrac{2}{3} \frac{\sqrt{r_0 r_1} + (\sqrt{r_0} + \sqrt{r_1})^2}{\sqrt{r_0 r_1} (\sqrt{r_0} + \sqrt{r_1})^3} \,
g^{\mu\nu} 
E_{r_0} (\hnabla_\mu \hnabla_\nu u) E_{r_1}
\\[2mm]
& + \tfrac{2}{3} \frac{Q_3(r_0, r_1, r_2)}{\sqrt{r_1} (\sqrt{r_0} + \sqrt{r_1})^2  (\sqrt{r_0} + \sqrt{r_2})^3 (\sqrt{r_1} + \sqrt{r_2})^2} \,
g^{\mu\nu} 
E_{r_0} (\hnabla_\mu u) E_{r_1} (\hnabla_\nu u) E_{r_2}
\\[2mm]
& + \tfrac{2}{3} Q_4(r_0, r_1, r_2) \, 
E_{r_0} (\hnabla_\mu u) E_{r_1} p^\mu E_{r_2}
- \tfrac{2}{3} Q_4(r_2, r_1, r_0) \, 
E_{r_0} p_\mu E_{r_1} p^\mu E_{r_2}
\\[2mm]
& - \tfrac{2}{3} \frac{\sqrt{r_0} + \sqrt{r_1} + \sqrt{r_2}}{\sqrt{r_0 r_1 r_2} (\sqrt{r_0} + \sqrt{r_1}) (\sqrt{r_0} + \sqrt{r_2}) (\sqrt{r_1} + \sqrt{r_2})} \, 
E_{r_0} p_\mu E_{r_1} p^\mu E_{r_2}
\big].
\end{aligned}
\end{align*}
\end{corollary}

To conclude this list of results given at various dimensions, we see that we have explicit expressions of $\mathcal{R}_2$ for $d$ even, and moreover simple generic expressions for $d = 2m$, $d \geq 4$, while for $d$ odd it is difficult to propose a generic expression.

\section{Some direct applications}
\label{sec-direct-applications}

\subsection{\texorpdfstring{The case $a = \bbbone$}{The case a=1}}

When $a = \bbbone$, a cyclic operation can be performed under the trace:
\begin{align*}
f(r_0, r_1, \dots, r_k) \tr( E_{r_0} B_1 E_{r_1} B_2 \cdots B_k E_{r_k})
&=
f(r_0, r_1, \dots, r_k) \tr( E_{r_k} E_{r_0} B_1 E_{r_1} B_2 \cdots B_k )
\\
& = \delta_{r_0, r_k} f(r_0, r_1, \dots, r_k) \tr( E_{r_k} E_{r_0} B_1 E_{r_1} B_2 \cdots B_k )
\\
& = f(r_0, r_1, \dots, r_0) \tr( E_{r_0} B_1 E_{r_1} B_2 \cdots E_{r_{k-1}} B_k ).
\end{align*} 
This implies that in all the spectral functions one can put $r_k = r_0$ (remember that all the spectral functions are continuous, so that $r_k \to r_0$ is well-defined). 

In \cite[Theorem~4.3]{IochMass17a}, $a_2(\bbbone, P)$ has been computed for $d=2m$ even, $m \geq 1$. Let us first rewrite this result in terms of spectral functions:
\begin{align*}
&\alpha \, r_0^{-m+1} \, E_0
+ r_0^{-m} \, E_0 w
+ \tfrac{m-2}{6}[ \tfrac{1}{2} \alpha^\mu - \beta^\mu ] \, r_0^{-m} \, E_0 (\pmu u)
- \tfrac{m-2}{6} g^{\mu\nu} \, r_0^{-m} \, E_0 (\pmu \pnu u)
\\
& + \tfrac{1}{2} \beta_\mu \, r_0^{-m} \, E_0 v^\mu
- \tfrac{1}{2} \, r_0^{-m} \, E_0 (\pmu v^\mu)
- \tfrac{1}{4m} g_{\mu\nu} \sum_{\ell=0}^{m-1} r_0^{-\ell-1} r_1^{\ell-m} \, E_0 v^\mu E_1 v^\nu
\\
& + \tfrac{1}{2m} \sum_{\ell=0}^{m-1} (m - 2\ell) \, r_0^{-\ell-1}  r_1^{\ell-m} \, E_0 v^\mu E_1 (\pmu u)
+ g^{\mu\nu} \sum_{\ell=0}^{m-1} [ \tfrac{m - 2}{6} - \tfrac{\ell (m - \ell - 1)}{2 m} ] \, r_0^{-\ell-1}  r_1^{\ell-m} \, E_0 (\pmu u) E_1 (\pnu u).
\end{align*}
Notice that:
\begin{align*}
g_{\mu\nu} \!\sum_{\ell=0}^{m-1} r_0^{-\ell-1} r_1^{\ell-m} \, \tr(E_0 v^\mu E_1 v^\nu)
&=
g_{\mu\nu} \!\sum_{\ell=0}^{m-1} r_0^{-\ell-1} r_1^{\ell-m} \, \tr(E_1 v^\nu E_0 v^\mu)
= g_{\mu\nu}\! \sum_{\ell=0}^{m-1} r_1^{-\ell-1} r_0^{\ell-m} \, \tr(E_1 v^\mu E_0 v^\nu )
\end{align*}
if we change $\ell$ to $m - \ell - 1$ in the summation and we use the symmetry of $g_{\mu\nu}$. Then, to show that the spectral functions $F$ reduce to the ones above, one has to use the symmetry $r_0 \leftrightarrow r_1$ for some terms. A direct computation gives
\begin{align*}
F_{w}(r_0, r_0)
&= r_0^{-m},
\\
F_{\partial u}^\mu(r_0, r_0)
&= \tfrac{m-2}{6} [ \tfrac{1}{2} \alpha^\mu - \beta^\mu ] \, r_0^{-m},
\\
F_{\partial \partial u}(r_0, r_0)
&= - \tfrac{m-2}{6} \, r_0^{-m},
\\
F_{v, \mu}(r_0, r_0)
&= \tfrac{1}{2} \beta_\mu \, r_0^{-m},
\\
F_{\partial v}(r_0, r_0)
&= - \tfrac{1}{2} \, r_0^{-m},
\\
\tfrac{1}{2} [ F_{v,v}(r_0, r_1, r_0) + F_{v,v}(r_1, r_0, r_1) ]
&= - \tfrac{1}{4m} \sum_{\ell=0}^{m-1} r_0^{-\ell-1} r_1^{\ell-m},
\\
\tfrac{1}{2} [ F_{\partial u, \partial u}(r_0, r_1, r_0) + F_{\partial u, \partial u}(r_1, r_0, r_1) ]
&= \sum_{\ell=0}^{m-1} [ \tfrac{m - 2}{6} - \tfrac{\ell (m - \ell - 1)}{2 m} ] \, r_0^{-\ell-1}  r_1^{\ell-m},
\\
F_{v,\partial u}(r_0, r_1, r_0) + F_{\partial u,v}(r_1, r_0, r_1)
&= \tfrac{1}{2m} \sum_{\ell=0}^{m-1} (m - 2\ell) \, r_0^{-\ell-1}  r_1^{\ell-m}.
\end{align*}
In the last expression, under the trace we have $\tr(E_0 (\pmu u) E_1 v^\mu) = \tr(E_1 v^\mu E_0 (\pmu u))$, and we need to sum the two functions $F_{v,\partial u}$ and $F_{\partial u,v}$ with their correct arguments.

These relations show that \eqref{eq-R-uvw-functions} reproduces \cite[eq.~4.15]{IochMass17a} when $a = \bbbone$.

\subsection{\texorpdfstring{Minimal Laplace type operators: $u = \bbbone$}{Minimal Laplace type operators: u = 1}}

Starting from \eqref{eq-def-P-upq} when $u = \bbbone$ one gets
\begin{align*}
P
&= - g^{\mu\nu} \nmu \nnu - ( p^\nu - [\tfrac{1}{2} \alpha^\nu - \beta^\nu ] ) \,\nnu - q
\end{align*}
and two simplifications occur in \eqref{eq-R-upq-functions}: all derivatives of $u$ vanish, and the spectrum of $u$ reduces to $\spec(u) = \{1\}$, so that all spectral functions are taken at $r_i = 1$ and $E_{r_i} = \bbbone$. The result is then
\begin{align*}
\mathcal{R}_2
&= \tfrac{1}{2^{d}\,\pi^{d/2}}\, [
\tfrac{1}{6} R
 + G_{q}(1, 1) \, 
q
+ G_{\hnabla p}(1, 1) \, 
\hnmu p^\mu
+ G_{p, p}(1, 1, 1) \, 
p^\mu  p_\mu 
\,].
\end{align*}
A direct computation gives $G_{q}(1, 1) = 1$, $G_{\hnabla p}(1, 1) = -\tfrac{1}{2}$, and $G_{p, p}(1, 1, 1) = -\tfrac{1}{4}$, so that
\begin{align}
\mathcal{R}_2
&= \tfrac{1}{2^{d}\,\pi^{d/2}} [
\tfrac{1}{6} R
 + q
-\tfrac{1}{2} \hnmu p^\mu
-\tfrac{1}{4} p^\mu  p_\mu 
].
\label{eq-u=1-pq}
\end{align}

As in Prop.~\ref{prop-change-connection-pq}, we can change $\hnmu$ to $\hnmu' \vc \hnmu +  \phi_\mu$ and solve $\phi_\mu$ in order to get $p^{\mu\,\prime} = 0$. From \eqref{eq-P-upq-hatnabla} with $u =  \bbbone$, one has
\begin{align*}
g^{\mu\nu} \hnmu' \hnnu' + q' &= g^{\mu\nu} (\hnmu +  \phi_\mu) \,(\hnnu +  \phi_\nu) + q' = g^{\mu\nu} \hnmu \hnnu + g^{\mu\nu} (\hnmu \phi_\nu) +  2 g^{\mu\nu} \phi_\nu \hnmu + g^{\mu\nu} \phi_\mu \phi_\nu + q' \\&\cv g^{\mu\nu}  \hnmu \hnnu  + p^\mu \hnmu +q
\end{align*}
with $p^\mu = 2 g^{\mu\nu} \phi_\nu$ and $q = q' + g^{\mu\nu} (\hnmu \phi_\nu) + g^{\mu\nu} \phi_\mu \phi_\nu$. This is solved for $\phi_\mu = \tfrac{1}{2} g_{\mu\nu} p^\nu$ and implies $q' = q - \tfrac{1}{2} \hnmu p^\mu -\tfrac{1}{4} p^\mu  p_\mu$. Injected into \eqref{eq-u=1-pq}, this gives $\mathcal{R}_2 =  \tfrac{1}{2^{d}\,\pi^{d/2}} [ \tfrac{1}{6} R + q' ]$ as in \cite[Theorem~3.3.1]{Gilk03a}.

\subsection{Conformal like transformed Laplacian}
\label{subsec-conf-like-lap}

Let us consider a positive invertible element $k \in \Gamma(\End(V))$, a covariant derivative $\nmu$ on $V$ and $\Delta = -g^{\mu\nu} \hnmu \hnnu$ be its associated Laplacian. We consider the operator
\begin{equation*}
P \vc k \Delta k = 
- g^{\mu\nu} k^2 \hnmu \hnnu - 2 g^{\mu\nu} k (\hnmu k) \, \hnnu + k (\Delta k).
\end{equation*}
It can be written as in \eqref{eq-P-upq-hatnabla} with
\begin{align*}
u = k^2,
\qquad
p^\nu = g^{\mu\nu} k (\hnmu k) - g^{\mu\nu} (\hnmu k) k ,
\qquad
q = -k (\Delta k).
\end{align*}
Application of Theorem~\ref{thm-R-upq} gives
\begin{align*}
\mathcal{R}_2^{k\Delta k}
&= \tfrac{1}{2^{d}\,\pi^{d/2}} \big[
\begin{multlined}[t]
\tfrac{1}{6} R \, r_0^{-d/2+1} \, E_{r_0}
\\
+
F_{\Delta k}^{k\Delta k}(r_0, r_1) \, E_{r_0} (\Delta k) E_{r_1}
+
F_{\nabla k \nabla k}^{k\Delta k}(r_0, r_1, r_2) g^{\mu\nu} \, E_{r_0} (\nmu k) E_{r_1} (\nnu k) E_{r_2}
\big]
\end{multlined}
\end{align*}
with
\begin{align}
F_{\Delta k}^{k\Delta k}(r_0, r_1)
&=
-\sqrt{r_0} \, G_{q}(r_0, r_1)
- (\sqrt{r_0} + \sqrt{r_1}) \, G_{\hnabla \hnabla u}(r_0, r_1)
- (\sqrt{r_0} - \sqrt{r_1}) \, G_{\hnabla p}(r_0, r_1),
\label{eq-conf-like-FDelta}
\\
F_{\nabla k \nabla k}^{k\Delta k}(r_0, r_1, r_2)
&= \begin{aligned}[t]
& 2 G_{\hnabla \hnabla u}(r_0, r_2) 
\\
& + (\sqrt{r_0} + \sqrt{r_1})(\sqrt{r_1} + \sqrt{r_2}) \, G_{\hnabla u, \hnabla u}(r_0, r_1, r_2)
\\
& +  (\sqrt{r_0} - \sqrt{r_1})(\sqrt{r_1} + \sqrt{r_2}) \, G_{p, \hnabla u}(r_0, r_1, r_2)
\\
& + (\sqrt{r_0} + \sqrt{r_1})(\sqrt{r_1} - \sqrt{r_2}) \, G_{\hnabla u, p}(r_0, r_1, r_2)
\\
& + (\sqrt{r_0} - \sqrt{r_1})(\sqrt{r_1} - \sqrt{r_2}) \, G_{p, p}(r_0, r_1, r_2).
\end{aligned}
\label{eq-conf-like-Fdkdk}
\end{align}
Using Proposition~\ref{prop-relations-G}, one has
\begin{align*}
F_{\Delta k}^{k\Delta k}(r_0, r_1)
&= 
- \frac{\sqrt{r_0 r_1}\, ( \sqrt{r_0} + \sqrt{r_1}) [ G_{q}(r_0, r_1) + 2 G_{\hnabla \hnabla u}(r_0, r_1)]}{r_0 + r_1}.
\end{align*}

\section{Details of the computations}
\label{sec-details-computations}

In this section we give some details on the computations to establish Theorems~\ref{thm-R-uvw} and \ref{thm-R-upq}. These computations can be done by hand but the reader can also follow \cite{IochMass17b}.

The equation of \eqref{eq-R-uvw-functions} requires to compute the terms in the sum \eqref{eq-R-Tkp}, which itself requires to compute the arguments $B_1 \otimes \cdots \otimes B_k$ and the operators $T_{k,p}$.

For $r=2$, the list of arguments has been evaluated in \cite[Section~4.1]{IochMass17a} (starting with the expression \eqref{eq-def-P-uvw} of $P$) as well as their contractions with the tensor $G(g)_{\mu_1\dots \mu_{2p}}$. We make use of these results below. Then the computation of the operators $T_{k,p}$ reduces to the computation of the universal spectral functions $I_{d/2+p,k}$. As noticed in \cite{IochMass17a}, only the values $k=1,2,3,4$ have to be considered. 

Below is the list of the evaluation of these arguments in the corresponding operators $T_{k,p}$, where the following functional calculus rule (and its obvious generalizations) is used
\begin{align*}
f(r_0, r_1, \dots, r_k) \, E_{r_0} u^{n_0} E_{r_1} u^{n_1} E_{r_2} \cdots u^{n_{k-1}} E_{r_k}
&=
r_0^{n_0 + n_1 + \cdots + n_{k-1}} f(r_0, r_0, \dots, r_0) \, E_{r_0},
\end{align*}
where summations over $r_i$ in the spectrum of $u$ are omitted. In the following, the symbol $\leadsto$ is used to symbolize this evaluation.

For $k=1$, there is only one argument:
\begin{align*}
w & \leadsto 
I_{d/2,1}(r_0, r_1)\, 
E_{r_0} w E_{r_1}
\end{align*}

For $k=2$, one has:
\begin{align*}
-\tfrac{1}{2} g^{\mu\nu} g_{\rho\sigma}(\pmu\pnu g^{\rho\sigma})\, u\otimes u 
& \leadsto
-\tfrac{1}{2} g^{\mu\nu} g_{\rho\sigma}(\pmu\pnu g^{\rho\sigma})\, 
r_0^2 I_{d/2+1,2}(r_0, r_0, r_0)\,
E_{r_0}\,,
\\[1mm]
-g^{\mu\nu} g_{\rho\sigma}(\pnu g^{\rho\sigma}) \,u\otimes \pmu u 
& \leadsto
-g^{\mu\nu} g_{\rho\sigma}(\pnu g^{\rho\sigma}) \,
r_0 I_{d/2+1,2}(r_0, r_0, r_1)\,
E_{r_0} (\pmu u) E_{r_1}\,,
\\[1mm]
-\tfrac{d}{2} g^{\mu\nu} \, u \otimes \pmu\pnu u 
& \leadsto
- \tfrac{d}{2} g^{\mu\nu} \, 
r_0 I_{d/2+1,2}(r_0, r_0, r_1)\,
E_{r_0} (\pmu \pnu u) E_{r_1}\,,
\\[1mm]
-\tfrac{1}{2} g_{\rho\sigma} (\pmu g^{\rho\sigma})\, v^\mu \otimes u
& \leadsto
-\tfrac{1}{2} g_{\rho\sigma} (\pmu g^{\rho\sigma})\, 
r_1 I_{d/2+1,2}(r_0, r_1, r_1)\,
E_{r_0} v^\mu E_{r_1}\,,
\\[1mm]
- \tfrac{d}{2} \,v^\mu \otimes \pmu u 
& \leadsto
- \tfrac{d}{2} \,
I_{d/2+1,2}(r_0, r_1, r_2)\,
E_{r_0} v^\mu E_{r_1} (\pmu u) E_{r_2}\,,
\\[1mm]
 -\tfrac{1}{2} g_{\mu\nu} \,v^\mu \otimes v^\nu
 & \leadsto
 -\tfrac{1}{2} g_{\mu\nu} \,
 I_{d/2+1,2}(r_0, r_1, r_2)\,
E_{r_0} v^\mu E_{r_1} v^\nu E_{r_2}\,,
\\[1mm]
 - u \otimes \pmu v^\mu
& \leadsto
 -\,
r_0 I_{d/2+1,2}(r_0, r_0, r_1)\,
E_{r_0} (\pmu v^\mu) E_{r_1}\,.
\end{align*}

For $k=3$, one has:
\begin{align*}
& \big[
\begin{aligned}[t]
& g^{\mu\nu} g_{\rho\sigma} (\pmu\pnu g^{\rho\sigma}) 
+ 2 (\pmu\pnu g^{\mu\nu})
+ g_{\rho\sigma} (\pmu g^{\mu\nu})(\pnu g^{\rho\sigma})
+ 2 g_{\rho\sigma}(\pmu g^{\nu\rho})(\pnu g^{\mu\sigma})
\\
& + \tfrac{1}{2} g^{\mu\nu} g_{\rho\sigma} g_{\alpha\beta} (\pmu g^{\rho\sigma})(\pnu g^{\alpha\beta}) 
+ g^{\mu\nu} g_{\rho\sigma} g_{\alpha\beta} (\pmu g^{\rho\alpha})(\pnu g^{\sigma\beta})
\big]
\,u\otimes u\otimes u 
\end{aligned}
\\
& \qquad\leadsto
\big[
\begin{aligned}[t]
& g^{\mu\nu} g_{\rho\sigma} (\pmu\pnu g^{\rho\sigma}) 
+ 2 (\pmu\pnu g^{\mu\nu})
+ g_{\rho\sigma} (\pmu g^{\mu\nu})(\pnu g^{\rho\sigma})
+ 2 g_{\rho\sigma}(\pmu g^{\nu\rho})(\pnu g^{\mu\sigma})
\\
& + \tfrac{1}{2} g^{\mu\nu} g_{\rho\sigma} g_{\alpha\beta} (\pmu g^{\rho\sigma})(\pnu g^{\alpha\beta}) 
+ g^{\mu\nu} g_{\rho\sigma} g_{\alpha\beta} (\pmu g^{\rho\alpha})(\pnu g^{\sigma\beta})
\big]\, 
r_0^3 I_{d/2+2,3}(r_0, r_0, r_0, r_0)\,
E_{r_0}\,,
\end{aligned}
\\[1mm]
& (d+6) [ \tfrac{1}{2} g^{\mu\nu}g_{\rho\sigma} (\pnu g^{\rho\sigma}) + (\pnu g^{\mu\nu})]
\,u\otimes u\otimes \pmu u 
\\
& \qquad\leadsto
(d+6) [ \tfrac{1}{2} g^{\mu\nu}g_{\rho\sigma} (\pnu g^{\rho\sigma}) + (\pnu g^{\mu\nu}) ]\,
r_0^2 I_{d/2+2,3}(r_0, r_0, r_0, r_1)\,
E_{r_0} (\pmu u) E_{r_1}\,,
\\[1mm]
&[
\tfrac{d+4}{2} g^{\mu\nu}g_{\rho\sigma} (\pnu g^{\rho\sigma}) + 2 (\pnu g^{\mu\nu})
]
\,u\otimes \pmu u\otimes u 
\\
& \qquad\leadsto
[ \tfrac{d+4}{2} g^{\mu\nu}g_{\rho\sigma} (\pnu g^{\rho\sigma}) + 2 (\pnu g^{\mu\nu}) ]\,
r_0 r_1 I_{d/2+2,3}(r_0, r_0, r_1, r_1)\,
E_{r_0} (\pmu u) E_{r_1}\,,
\\[1mm]
&\tfrac{(d+2)^2}{2} g^{\mu\nu} \,u\otimes \pmu u\otimes \pnu u 
\leadsto
\tfrac{(d+2)^2}{2} g^{\mu\nu} \,
r_0 I_{d/2+2,3}(r_0, r_0, r_1, r_2)\,
E_{r_0} (\pmu u) E_{r_1} (\pnu u) E_{r_2}\,,
\\[1mm]
& (d+2) g^{\mu\nu}\, u\otimes u \otimes \pmu\pnu u
\leadsto
(d+2) g^{\mu\nu}\, 
r_0^2 I_{d/2+2,3}(r_0, r_0, r_0, r_1)\,
E_{r_0} (\pmu \pnu u) E_{r_1}\,,
\\[1mm]
& [\tfrac{1}{2} g_{\rho\sigma} (\pmu g^{\rho\sigma}) + g_{\mu\nu}(\prho g^{\rho\nu})]
(v^\mu \otimes u \otimes u + u \otimes v^\mu \otimes u + u \otimes u \otimes v^\mu)
\\
& \qquad\leadsto
\begin{multlined}[t]
[\tfrac{1}{2} g_{\rho\sigma} (\pmu g^{\rho\sigma}) + g_{\mu\nu}(\prho g^{\rho\nu})]\,
\big[ r_1^2 I_{d/2+2, 3}(r_0, r_1, r_1, r_1) \\
+ r_0 r_1 I_{d/2+2, 3}(r_0, r_0, r_1, r_1) + r_0^2 I_{d/2+2, 3}(r_0, r_0, r_0, r_1)\big]\,
E_{r_0} v^\mu E_{r_1}\,,
\end{multlined}
\\[1mm]
& \tfrac{d+2}{2} (v^\mu \otimes u \otimes \pmu u)
\leadsto
\tfrac{d+2}{2}\,
r_1 I_{d/2+2, 3}(r_0, r_1, r_1, r_2)\,
E_{r_0} v^\mu E_{r_1} (\pmu u) E_{r_2}\,,
\\[1mm]
& \tfrac{d+2}{2} (u \otimes \pmu u \otimes v^\mu)
\leadsto
\tfrac{d+2}{2}\,
r_0 I_{d/2+2, 3}(r_0, r_0, r_1, r_2)\,
E_{r_0} (\pmu u) E_{r_1} v^\mu E_{r_2}\,,
\\[1mm]
& \tfrac{d+2}{2} (u \otimes v^\mu \otimes \pmu u)
 \leadsto
\tfrac{d+2}{2}\,
r_0 I_{d/2+2, 3}(r_0, r_0, r_1, r_2)\,
E_{r_0} v^\mu E_{r_1} (\pmu u) E_{r_2}\,.
\end{align*}

For $k=4$, one has:
\begin{align*}
&3\big[
\begin{aligned}[t]
& -\tfrac{1}{2} g^{\mu\nu} g_{\rho\sigma} g_{\alpha\beta} (\pmu g^{\rho\sigma}) (\pnu g^{\alpha\beta}) 
- g^{\mu\nu} g_{\rho\sigma} g_{\alpha\beta} (\pmu g^{\rho\alpha}) (\pnu g^{\sigma\beta})
-2 g_{\rho\sigma} (\pmu g^{\mu\nu}) (\pnu g^{\rho\sigma})
\\
& -2 g_{\rho\sigma} (\pmu g^{\mu\rho}) (\pnu g^{\nu\sigma})
-2 g_{\rho\sigma} (\pmu g^{\nu\rho}) (\pnu g^{\mu\sigma})
\big] \, u \otimes u \otimes u \otimes u 
\end{aligned}
\\
& \qquad\leadsto
3\big[
\begin{aligned}[t]
& -\tfrac{1}{2} g^{\mu\nu} g_{\rho\sigma} g_{\alpha\beta} (\pmu g^{\rho\sigma}) (\pnu g^{\alpha\beta}) 
- g^{\mu\nu} g_{\rho\sigma} g_{\alpha\beta} (\pmu g^{\rho\alpha}) (\pnu g^{\sigma\beta})
-2 g_{\rho\sigma} (\pmu g^{\mu\nu}) (\pnu g^{\rho\sigma})
\\
& -2 g_{\rho\sigma} (\pmu g^{\mu\rho}) (\pnu g^{\nu\sigma})
-2 g_{\rho\sigma} (\pmu g^{\nu\rho}) (\pnu g^{\mu\sigma})
\big] \,
r_0^4 I_{d/2+3, 4}(r_0, r_0, r_0, r_0, r_0)\,
E_{r_0}\,,
\end{aligned}
\\[1mm]
& \begin{multlined}[b]
- (d+4) [\tfrac{1}{2} g^{\mu\nu} g_{\rho\sigma}(\pnu g^{\rho\sigma}) + (\pnu g^{\mu\nu})] 
[ 
3 \, u\otimes u \otimes u\otimes \pmu u 
\\ 
+ 2 \,u\otimes u \otimes \pmu u\otimes u
+ u \otimes \pmu u \otimes u \otimes u ]
\end{multlined}
\\
& \qquad\leadsto
\begin{multlined}[t]
- (d+4) [\tfrac{1}{2} g^{\mu\nu} g_{\rho\sigma}(\pnu g^{\rho\sigma}) + (\pnu g^{\mu\nu}) ]
[ 3 r_0^3 I_{d/2+3, 4}(r_0, r_0, r_0, r_0, r_1)
\\
+ 2 r_0^2 r_1 I_{d/2+3, 4}(r_0, r_0, r_0, r_1, r_1)
+ r_0 r_1^2 I_{d/2+3, 4}(r_0, r_0, r_1, r_1, r_1) ]
E_{r_0} (\pmu u) E_{r_1}\,,
\end{multlined}
\\[1mm]
& -\tfrac{1}{2} (d+4)(d+2) g^{\mu\nu} \, 
(2\,u \otimes u \otimes \pmu u \otimes \pnu u + u \otimes \pmu u \otimes u \otimes \pnu u )
\\
& \qquad\leadsto
\begin{multlined}[t]
 - \tfrac{1}{2} (d+4)(d+2) g^{\mu\nu} \, 
[ 2 r_0^2 I_{d/2+3, 4}(r_0, r_0, r_0, r_1, r_2) + r_0 r_1 I_{d/2+3, 4}(r_0, r_0, r_1, r_1, r_2) ] 
\\
\! E_{r_0} (\pmu u) E_{r_1} (\pnu u) E_{r_2}\,.
\end{multlined}
\end{align*}
The coefficient $\alpha$ and the spectral functions $F$ are evaluated by collecting these terms:
\begin{align}
\alpha 
&\vc 
\begin{aligned}[t]
& \tfrac{1}{3} (\pmu\pnu g^{\mu\nu}) 
-\tfrac{1}{12} g^{\mu\nu} g_{\rho \sigma}(\pmu\pnu g^{\rho\sigma}) 
+ \tfrac{1}{48} g^{\mu\nu} g_{\rho \sigma} g_{\alpha\beta}(\pmu g^{\rho\sigma})(\pnu g^{\alpha\beta})
\\[1mm]
& + \tfrac{1}{24} g^{\mu\nu} g_{\rho \sigma} g_{\alpha\beta}(\pmu g^{\rho\alpha})(\pnu g^{\sigma\beta}) 
-\tfrac{1}{12} g_{\rho\sigma}(\pmu g^{\mu\nu})(\pnu g^{\rho\sigma}) 
\\[1mm]
& + \tfrac{1}{12} g_{\rho\sigma}(\pmu g^{\nu\rho})(\pnu g^{\mu\sigma}) 
- \tfrac{1}{4} g_{\rho\sigma}(\pmu g^{\mu\rho})(\pnu g^{\nu\sigma}),
\end{aligned}
\nonumber
\\
F_{w}(r_0, r_1)
&\vc
I_{d/2,1}(r_0, r_1),
\nonumber
\\
F_{\partial v}(r_0, r_1)
&\vc
 -\, r_0 I_{d/2+1,2}(r_0, r_0, r_1), 
\nonumber
\\
F_{\partial \partial u}(r_0, r_1)
&\vc
 - \tfrac{d}{2} \, r_0 I_{d/2+1,2}(r_0, r_0, r_1) 
+ (d+2) \, r_0^2 I_{d/2+2,3}(r_0, r_0, r_0, r_1),
\nonumber
\\
F_{\partial u}^\mu(r_0, r_1)
&\vc \begin{aligned}[t]
& - \alpha^\mu \,
r_0 I_{d/2+1,2}(r_0, r_0, r_1)\, 
+ (d+6) [ \tfrac{1}{2} \alpha^\mu + \beta^\mu ]\,
r_0^2 I_{d/2+2,3}(r_0, r_0, r_0, r_1)\, 
\\[1mm]
& + [ \tfrac{d+4}{2} \alpha^\mu + 2 \beta^\mu ]\,
r_0 r_1 I_{d/2+2,3}(r_0, r_0, r_1, r_1)\, 
\\[1mm]
& \begin{multlined}[t]
- (d+4) [\tfrac{1}{2} \alpha^\mu + \beta^\mu ]
[ 3 r_0^3 I_{d/2+3, 4}(r_0, r_0, r_0, r_0, r_1)
\\[1mm]
+ 2 r_0^2 r_1 I_{d/2+3, 4}(r_0, r_0, r_0, r_1, r_1)
+ r_0 r_1^2 I_{d/2+3, 4}(r_0, r_0, r_1, r_1, r_1) ], 
\end{multlined}
\end{aligned}
\nonumber
\\
F_{v, \mu}(r_0, r_1)
&\vc 
\begin{multlined}[t]
-\tfrac{1}{2} \alpha_\mu\, 
r_1 I_{d/2+1,2}(r_0, r_1, r_1)\, 
+
[\tfrac{1}{2} \alpha_\mu + \beta_\mu ]\,
[ r_1^2 I_{d/2+2, 3}(r_0, r_1, r_1, r_1) 
\\
+ r_0 r_1 I_{d/2+2, 3}(r_0, r_0, r_1, r_1) 
+ r_0^2 I_{d/2+2, 3}(r_0, r_0, r_0, r_1)],
\end{multlined}
\nonumber
\\
F_{v,v}(r_0, r_1, r_2)
&\vc
-\tfrac{1}{2} \, I_{d/2+1,2}(r_0, r_1, r_2),
\nonumber
\\
F_{\partial u,v}(r_0, r_1, r_2)
&\vc
\tfrac{d+2}{2} \, r_0 I_{d/2+2, 3}(r_0, r_0, r_1, r_2),
\nonumber
\\
F_{v,\partial u}(r_0, r_1, r_2)
&\vc \begin{aligned}[t]
& - \tfrac{d}{2} \,
I_{d/2+1,2}(r_0, r_1, r_2)
+ \tfrac{d+2}{2} \,
r_1 I_{d/2+2, 3}(r_0, r_1, r_1, r_2)
\\
& + \tfrac{d+2}{2} \,
r_0 I_{d/2+2, 3}(r_0, r_0, r_1, r_2),
\end{aligned}
\nonumber
\\
F_{\partial u, \partial u}(r_0, r_1, r_2)
&\vc
\begin{aligned}[t]
& \tfrac{(d+2)^2}{2} \,
r_0 I_{d/2+2,3}(r_0, r_0, r_1, r_2)
\\
& \begin{multlined}[t]
 - \tfrac{(d+4)(d+2)}{2} \, 
[ 2 r_0^2 I_{d/2+3, 4}(r_0, r_0, r_0, r_1, r_2) 
+ r_0 r_1 I_{d/2+3, 4}(r_0, r_0, r_1, r_1, r_2) ].
\end{multlined}
\end{aligned}
\label{eq-def-F-I}
\end{align}

\bigskip
As in \cite[Section~4.3]{IochMass17a}, the strategy to compute \eqref{eq-R-upq-functions} could be to make the change of variables $(u, v^\mu, w) \mapsto (u, p^\mu, q)$. Here we use another strategy which simplifies the computation since it is based on \eqref{eq-P-exp-HKP-f}, \eqref{eq-H-K-upq}, and \eqref{eq-change-uvw-upq}. 

Indeed, as already noticed, one can apply verbatim the computation of the arguments and their contractions with $\pmu$ replaced by $\nmu$, and at the same time, using $p^\mu + g^{\mu\nu} (\nnu u) - [\tfrac{1}{2} \alpha^\mu - \beta^\mu ] u$ in place of $v^\mu$, and $q$ in place of $w$. So, \eqref{eq-R-uvw-functions} can be replaced by
\begin{align*}
& \alpha\, r_0^{-d/2+1}\, 
E_{r_0}
+ F_{\partial u}^\mu(r_0, r_1) \, 
E_{r_0} (\nmu u) E_{r_1}
+ g^{\mu\nu} F_{\partial \partial u}(r_0, r_1) \, 
E_{r_0} (\nmu \nnu u) E_{r_1}
\\
& + g^{\mu\nu} F_{\partial u, \partial u}(r_0, r_1, r_2) \, 
E_{r_0} (\nmu u) E_{r_1} (\nnu u) E_{r_2}
+ F_{w}(r_0, r_1) \, 
E_{r_0} w E_{r_1}
+ F_{v, \mu}(r_0, r_1) \, 
E_{r_0} v^\mu E_{r_1}
\\
& + F_{v,\partial u}(r_0, r_1, r_2) \, 
E_{r_0} v^\mu E_{r_1} (\nmu u) E_{r_2}
+ F_{\partial u,v}(r_0, r_1, r_2) \, 
a E_{r_0} (\nmu u) E_{r_1} v^\mu E_{r_2}
\\
& + g_{\mu\nu} F_{v,v}(r_0, r_1, r_2) \, 
E_{r_0} v^\mu E_{r_1} v^\nu E_{r_2}
+ F_{\partial v}(r_0, r_1) \, 
E_{r_0} (\nmu v^\mu) E_{r_1}.
\label{eq-uvwnabla-functions}
\end{align*}
The next step is to replace $\nmu$ by $\hnmu$. We use  \eqref{eq-nablaLCu}, \eqref{eq-laplacianLCu} and
\begin{align*}
\nmu v^\mu
&= \nmu p^\mu 
+ (\pmu g^{\mu\nu}) (\nnu u)
+ g^{\mu\nu} (\nmu \nnu u)
- [\tfrac{1}{2} \pmu \alpha^\mu - \pmu \beta^\mu ] u
- [\tfrac{1}{2} \alpha^\mu - \beta^\mu ] (\nmu u)
\\
&= \nmu p^\mu 
+ g^{\mu\nu} (\hnabla_\mu \hnabla_\nu u)
+ \beta^\mu (\hnabla_\mu u)
- [\tfrac{1}{2} \pmu \alpha^\mu - \pmu \beta^\mu ] u
\\
&= \hnabla_\mu p^\mu 
+ \tfrac{1}{2} \alpha_\mu p^\mu
+ g^{\mu\nu} (\hnabla_\mu \hnabla_\nu u)
+ \beta^\mu (\hnabla_\mu u)
- [\tfrac{1}{2} \pmu \alpha^\mu - \pmu \beta^\mu ] u.
\end{align*}
This leads to a new expression containing the terms in \eqref{eq-R-upq-functions}, with the functions given in the list after Theorem~\ref{thm-R-upq}, to which we have to add the two terms $G_{\hnabla u}^\mu(r_0, r_1)\, E_{r_0} (\hnmu u) E_{r_1} + G_{p, \mu}(r_0, r_1) E_{r_0} p^\mu E_{r_1}$, with 
\begin{align*}
G_{\hnabla u}^\mu(r_0, r_1)
&= \begin{multlined}[t] 
F_{\partial u}^\mu(r_0, r_1)
+ g^{\mu\nu} \, F_{v, \nu}(r_0, r_1)
+ \beta^\mu  \, F_{\partial v}(r_0, r_1)
\\
 + [ \tfrac{1}{2} \alpha^\mu - \beta^\mu ] 
[ 
F_{\partial \partial u}(r_0, r_1)
 - r_0 F_{v,\partial u}(r_0, r_0, r_1)
 - r_1 F_{\partial u,v}(r_0, r_1, r_1)
 \\
- r_0 F_{v,v}(r_0, r_0, r_1)
- r_1 F_{v,v}(r_0, r_1, r_1)
],
\end{multlined}
\\
G_{p, \mu}(r_0, r_1)
&= \begin{multlined}[t]
F_{v, \mu}(r_0, r_1)
+ \tfrac{1}{2} \alpha_\mu F_{\partial v}(r_0, r_1)
\\
- [\tfrac{1}{2} \alpha_\mu - \beta_\mu ] \,
[
 r_0 F_{v,v}(r_0, r_0, r_1)
+ r_1 F_{v,v}(r_0, r_1, r_1)
].
\end{multlined}
\end{align*}
A direct computation performed using the expressions of the spectral functions $F$ in terms of $I_{d/2,1}$ shows that
\begin{equation*}
G_{\hnabla u}^\mu(r_0, r_1) = G_{p, \mu}(r_0, r_1) = 0.
\end{equation*}
The coefficient in front of $r_0^{-d/2+1} \, E_{r_0}$ is
\begin{align*}
\tfrac{1}{6} R
&=
\alpha 
- \tfrac{1}{4} \alpha^\mu \beta_\mu
+ \tfrac{1}{2} \beta^\mu \beta_\mu
+ \tfrac{1}{4} \pmu \alpha^\mu - \tfrac{1}{2} \pmu \beta^\mu
- \tfrac{1}{16} \alpha_\mu \alpha^\mu + \tfrac{1}{4} \alpha_\mu \beta^\mu - \tfrac{1}{4} \beta_\mu \beta^\mu,
\end{align*}
where $R$ is the scalar curvature of the metric $g$.

The spectral functions $G$ can be written in terms of log functions for $d=2$ (see \cite[Cor.~3.3]{IochMass17a}), as Laurent polynomials for $d \geq 4$ even (see \cite[Prop.~3.5]{IochMass17a}), and in terms of square roots of $r_i$ for $d$ odd (see \cite[Prop.~3.4]{IochMass17a}). This completes the proof of Corollaries~\ref{cor-R-d=2}, \ref{cor-R-d>2-even}, and \ref{cor-R-d=3}.

\section{Applications to the noncommutative torus}
\label{sec-applications-NCT}

In this section, we first apply Theorem~\ref{thm-R-upq} to the noncommutative 2-torus at \emph{rational values} of the deformation parameter $\theta$, for which it is known that we get a geometrical description in terms of sections of a fiber bundle. Some computations of $a_2(a,P)$ for specific operators $P$ have been performed at irrational values of $\theta$ to determine the so-called scalar curvature (our $\mathcal{R}_2$) \cite{ConnTret11a, ConnMosc14a, FathKhal11a, FathKhal12a, DabSit13, FathKhal13a, Sitarz14, Fath15a, DabSit15, Liu15a, Sade16a, ConnesFath16}. 

We now show that we can apply our general result at rational values of $\theta$ and get the same expressions for the scalar curvature $\mathcal{R}_2$ which appears to be written in terms of $\theta$-independent spectral functions. In particular, its expression is the same for rational and irrational $\theta$.

Let $\Theta \in M_d(\R)$ be a skew-symmetric real matrix. The noncommutative $d$-dimensional torus $C(\TT^d_\Theta)$ is defined as the universal unital $C^\ast$-algebra generated by unitaries $U_k$, $k = 1, \dots , d$, satisfying the relations
\begin{equation}
\label{eq-defrelationNCT}
U_k U_\ell = e^{2 i \pi \Theta_{k,\ell}} U_\ell U_k.
\end{equation}
This $C^\ast$-algebra contains, as a dense sub-algebra, the space of smooth elements for the natural action of the $d$-dimensional torus $\TT^d$ on $C(\TT^d_\Theta)$. This sub-algebra is described as elements in $C(\TT^d_\Theta)$ with an expansion
\begin{equation*}
a = \sum_{(k_i) \in \Z^{d}} a_{k_1, \dots, k_{d}} U_1^{k_1} \cdots U_{d}^{k_{d}}
\end{equation*}
where the sequence $(a_{k_1, \dots, k_{d}})$ belongs to the Schwartz space $\cS(\Z^{d})$. We denote by $C^\infty(\TT^d_\Theta)$ this algebra. The $C^\ast$-algebra $C(\TT^d_\Theta)$ has a unique normalized faithful positive trace $\NCtr$ whose restriction on smooth elements is given by
\begin{equation*}
\NCtr\, ( \sum_{(k_i) \in \Z^{d}} a_{k_1, \dots, k_{d}} U_1^{k_1} \cdots U_{d}^{k_{d}}) \vc a_{0, \dots, 0}
\end{equation*}
This trace satisfies $\NCtr\,(\bbbone) = 1$ where $\bbbone$ in the unit element of $C(\TT^d_\Theta)$. The smooth algebra  $C^\infty(\TT^d_\Theta)$ has $d$ canonical derivations $\dmu$, $\mu=1,\dots,d$, defined on the generators by
\begin{equation}
\label{def: deltak}
\dmu(U_k) \vc \delta_{\mu,k}\, i U_k.
\end{equation}
For any $a \in C(\TT^d_\Theta)$, one has $\dmu(a^\ast) = (\dmu a)^\ast$ (real derivations). 

Denote by $\H$ the Hilbert space of the GNS representation of $C(\TT^d_\Theta)$ defined by $\NCtr$. Each  derivation $\dmu$ defines a unbounded operator on $\H$, denoted also by $\dmu$, which satisfies $\dmu^\dagger = - \dmu$ (here $^\dagger$ denotes the adjoint of the operator).

\subsection{The geometry of the rational noncommutative tori}

In the following, we consider the special case  of even dimensional noncommutative tori, $d = 2m$, with
\begin{equation}
\label{def: matrice theta}
\Theta = 
\begin{pmatrix}
\theta_1 \chi & \cdots & 0 \\
0 & \ddots & 0 \\
0 & \cdots & \theta_m \chi
\end{pmatrix}, \text{ where } \chi \vc \begin{pmatrix} 0 & 1 \\ -1 & 0 \end{pmatrix},
\end{equation}
for a family of deformation parameters $\theta_1, \ldots, \theta_m$. Then
\begin{equation*}
C(\TT^{2m}_\Theta) \simeq C(\TT^{2}_{\Theta_1}) \otimes \cdots \otimes C(\TT^{2}_{\Theta_m})
\end{equation*}

When $d=2$ and $\theta = p/q$, where $p,q$ are relatively prime integers and $q>0$, it is known that $C(\TT^{2}_\Theta) \simeq \Gamma(A_{\theta})$ is isomorphic to the algebra of continuous sections of a fiber bundle $A_{\theta}$ in $M_q(\C)$ algebras over a $2$-torus $\TT^{2}_B$, as recalled in Section~\ref{sec-geometry-NCT}. Similarly, for $d=4$, with $\theta_1 = p_1/q_1$ and $\theta_2 = p_2/q_2$, $C(\TT^{4}_\Theta)$ is the space of sections of a fiber bundle in $M_{q_1 q_2}(\C)$ algebras over a $4$-torus $\TT^{2}_{B, 1} \times \TT^{2}_{B, 2}$. \\
Moreover, in the identification $C^\infty(\TT^2_\Theta) \simeq \Gamma^\infty(A_{\theta})$, the two derivations $\dmu$ are the two components of the unique flat connection $\nmu$ on $A_{\theta}$.

This geometrical description allows to use the results of Section~\ref{sec-method-results} to compute $a_2(a, P)$ for a differential operator on $\H$ of the form $P = - g^{\mu\nu} u \dmu \dnu - [ p^\nu + g^{\mu\nu} (\dmu u) - (\tfrac{1}{2} \alpha^\nu - \beta^\nu ) u ] \dnu - q$

\subsection{The noncommutative two torus}

In this section, we compute the coefficient $a_2(a, P)$ on the rational noncommutative two torus for a differential operator $P$ considered in \cite{ConnTret11a, FathKhal11a, FathKhal12a, ConnMosc14a} for the irrational noncommutative two torus. Let us introduce the following notations.

Let $\tau = \tau_1 + i \tau_2 \in \C$ with non zero imaginary part. We consider the constant metric $g$ defined by
\begin{align*}
g^{11} = 1,
\qquad
g^{12} = g^{21} &= \Re(\tau) = \tau_1,
\qquad
g^{22} = \abs{\tau}^2,
\end{align*}
with inverse matrix
\begin{align*}
g_{11} = \tfrac{\abs{\tau}^2}{\Im(\tau)^2} = \tfrac{\abs{\tau}^2}{\tau_2^2},
\qquad
g_{12} = g_{21} = - \tfrac{\Re(\tau)}{\Im(\tau)^2} = - \tfrac{\tau_1}{\tau_2^2},
\qquad
g_{22} = \tfrac{1}{\Im(\tau)^2} = \tfrac{1}{\tau_2^2}.
\end{align*}
We will use the constant tensors 
\begin{align*}
\epsilon^1 \vc 1,
\qquad
\epsilon^2 \vc \tau,
\qquad
\bepsilon^1 = 1,
\qquad
\bepsilon^2 = \btau,
\qquad
h^{\mu\nu} \vc \bepsilon^\mu \epsilon^\nu,
\end{align*}
which imply $h^{11} = 1$, $h^{12} = \tau$, $h^{21} = \btau$, $h^{22} = \abs{\tau}^2$. Then the symmetric part of $h^{\mu\nu}$ is the metric, $g^{\mu\nu} = \tfrac{1}{2} (h^{\mu\nu} + h^{\nu\mu})$, and $g_{\mu\nu} \bepsilon^\mu \bepsilon^\nu = 0$.

On the (GNS) Hilbert space $\H$, consider the following operators $\delta$, $\delta^\dagger$ and the Laplacian:
\begin{align*}
 \delta \vc \bepsilon^\mu \dmu = \delta_1 + \btau \delta_2,\qquad\delta^\dagger = - \epsilon^\mu \dmu = - \delta_1 - \tau \delta_2\\
\Delta \vc \delta^\dagger \delta 
= - \epsilon^\mu \bepsilon^\nu \dmu \dnu
= - h^{\mu\nu} \dmu \dnu
= - g^{\mu\nu} \dmu \dnu.
\end{align*}
For $k \in C^\infty(\TT^d_\Theta)$, $k>0$, the operator $P$ is defined as
\begin{align}
\label{def: P1 et P2}
P \vc \begin{pmatrix} P_1 & 0 \\ 0 & P_2 \end{pmatrix}
\end{align}
with
\begin{align*}
P_1 &\vc k \Delta k 
= - g^{\mu\nu} k \dmu \dnu k
=  - g^{\mu\nu} k^2 \dmu \dnu - 2 g^{\mu\nu} k (\dnu k) \dmu  - g^{\mu\nu} k (\dmu \dnu k)
\\
& \cv - u_1 g^{\mu\nu} \dmu \dnu - v_1^\mu \dmu - w_1,
\\[1mm]
P_2 &\vc \delta^\dagger k^2 \delta 
= - \epsilon^\nu \bepsilon^\mu \dnu k^2 \dmu
= - \bepsilon^\mu \epsilon^\nu (\dnu k^2) \dmu - \bepsilon^\mu \epsilon^\nu k^2 \dmu \dnu
= - g^{\mu\nu} k^2 \dmu \dnu - h^{\mu\nu} (\dnu k^2) \dmu
\\
& \cv - u_2 g^{\mu\nu} \dmu \dnu - v_2^\mu \dmu - w_2,
\end{align*}
so that
\begin{align*}
u_1 &= k^2,
&
v_1^\mu &= 2 g^{\mu\nu} k (\dnu k),
&
w_1 & = - k (\Delta  k),
\\
u_2 &= k^2,
&
v_2^\mu &= h^{\mu\nu} (\dnu k^2)
= h^{\mu\nu} [ k (\dnu k) + (\dnu k) k ],
&
w_2 & = 0.
\end{align*}
For the forthcoming computations, since the metric $g$ is constant, we have $\alpha = R = \alpha^\mu = \beta^\mu = 0$, $\hnmu = \nmu = \dmu$ (the last equality being a property of the geometrical presentation of $C^\infty(\TT^d_\Theta)$, as recalled above), $F_{\partial u}^\mu = F_{v, \mu} = 0$. Here $\abs{g}^{1/2} = \det(g_{\mu\nu})^{1/2}= \tau_2^{-1}$. We can write $P_1$ and $P_2$ in the covariant form \eqref{eq-def-P-upq} with
\begin{align*}
u_1 &= k^2,
&
p_1^\mu &= g^{\mu\nu} \big[ k (\dnu k) - (\dnu k) k \big],
&
q_1 & = - k (\Delta  k),
\\
u_2 &= k^2,
&
p_2^\mu &= (h^{\mu\nu} - g^{\mu\nu})\big[ k (\dnu k) + (\dnu k) k \big],
&
q_2 & = 0.
\end{align*}

Let 
\begin{align*}
f^{\mu\nu} &\vc \tfrac{1}{2} (h^{\mu\nu} - h^{\nu\mu}),
\end{align*}
so that $h^{\mu\nu} = g^{\mu\nu} + f^{\mu\nu}$ and $h^{\nu\mu} = g^{\mu\nu} - f^{\mu\nu}$, and define
\begin{align*}
Q_g(a,b,c)
&\vc
\frac{ \sqrt{a} 
(a \sqrt{b} + 3 a \sqrt{c} - \sqrt{a} c - \sqrt{a b c} - 2 b \sqrt{c})}
{(a - b) (\sqrt{a} - \sqrt{b})(\sqrt{a} - \sqrt{c})^3},
\\
Q_f(a,b,c)
&\vc
\frac{a (\sqrt{b} + \sqrt{c})}
{(a - b) (\sqrt{a} - \sqrt{b}) (a - c)},
\end{align*}
with the following (spectral) functions
\begin{align}
F_{\Delta k}(r_0, r_1) &=
\frac{r_0 - r_1 - \sqrt{r_0 r_1} \big( \ln(r_0) - \ln(r_1) \big)}{( \sqrt{r_0} - \sqrt{r_1} )^3},
\label{eq-NCT2-FDk}
\\[3mm]
F_{\partial k \partial k}^{\mu\nu}(r_0, r_1, r_2)
&= \begin{aligned}[t]
&
\frac{g^{\mu\nu} (\sqrt{r_0} + \sqrt{r_2}) (\sqrt{r_0} -2 \sqrt{r_1} + \sqrt{r_2})
+ f^{\mu\nu} (\sqrt{r_0} - \sqrt{r_2})^2}
{(\sqrt{r_0} - \sqrt{r_1}) (\sqrt{r_0} - \sqrt{r_2})^2 (\sqrt{r_1} - \sqrt{r_2})}
\\
&
+ [ g^{\mu\nu} Q_g(r_0, r_1, r_2) - f^{\mu\nu} Q_f(r_0, r_1, r_2) ] \ln(r_0)
\\
&
- \frac{( r_1 + \sqrt{r_0 r_2} )[ 
g^{\mu\nu} ( r_1 + \sqrt{r_0 r_2} ) - f^{\mu\nu} ( r_1 - \sqrt{r_0 r_2} )
 ]}
{(\sqrt{r_0} - \sqrt{r_1}) (r_0 - r_1)(\sqrt{r_1} - \sqrt{r_2}) (r_1 - r_2)} \ln(r_1)
\\
&
+ [ g^{\mu\nu} Q_g(r_2, r_1, r_0) - f^{\mu\nu} Q_f(r_2, r_1, r_0) ] \ln(r_2).
\end{aligned}
\label{eq-NCT2-Fdkdk}
\end{align}
With these notations:
\begin{proposition}
For the $2$-dimensional noncommutative torus at rational values of the deformation parameter $\theta$, one has $a_2(a,P) = \varphi \circ L(a \mathcal{R}_2)$ for any $a \in C(\TT^{2}_\Theta)$ with
\begin{align}
\label{eq-NCT-result}
\mathcal{R}_2
&\cv
\tfrac{1}{4\,\pi}\,[
F_{\Delta k}(r_0, r_1) \, \hlterm{E_{r_0} (\Delta k) E_{r_1}}
+ F_{\partial k \partial k}^{\mu\nu}(r_0, r_1, r_2) \, \hlterm{E_{r_0} (\dmu k) E_{r_1} (\dnu k) E_{r_2}}
\,].
\end{align}
where $L(a)$ is the left multiplication of $a$ on $\H$ (GNS representation).
\end{proposition}

Since we are in dimension $d=2$, the appearance of the log function in this result is expected. It is shown in Appendix~\ref{sec-comparison-NCT} that this result coincides with a previous one in \cite[Theorem~5.2]{FathKhal11a} for the irrational noncommutative two torus. As already mentioned, this computation, combined with results for irrational values of $\theta$, proves that $\mathcal{R}_2$ does not depend on the deformation parameter $\theta$, and in particular if it is irrational or not.  The simplification of the $q$'s in \eqref{eq-varphi-NCt-computation} (see also \eqref{eq-varphi-NCt-computation-d=4}) points out that the choice of the Hilbert space $\H$ as the GNS representation is a key fact because it permits to identify $\mathcal{R}_2$ in \eqref{eq-NCT-result} as an element in $C(\TT^{2}_\Theta)$ acting by left multiplication on $\H$. 

Nevertheless, the fact that $\mathcal{R}_2$ can be written in terms of $\theta$-independent spectral functions still needs an abstract proof. The spectrum of the differential operator $P$ depends on the differential operators $\dmu$ and some multiplication operators by elements of the algebra (written here in terms of $k$ and its derivatives $\dmu k$, $\dmu\dnu k$). On the one hand, the spectrum of the closed extension of the operator $\dmu$ in the Hilbert space of the GNS representation consists only of eigenvalues $ik_\mu$, $k_\mu \in \Z$, associated to  eigenvectors $U_1^{k_1} \cdots U_{d}^{k_{d}}$, so that it does not depend explicitly of $\theta$. On the other hand, the computations of $\mathcal{R}_2$, performed here or in \cite{ConnTret11a, FathKhal11a, FathKhal12a, ConnMosc14a}, are based on formal manipulations of the product in the algebra, in particular they do not use the defining relations \eqref{eq-defrelationNCT}. This explains why these methods bypass the $\theta$ dependency and give rise to some expressions in terms of $\theta$-independent spectral functions. Notice that for specific values of $\theta$, for instance $\theta = 0$ (the commutative case), these expressions can be simplified. So, one has to look at \eqref{eq-NCT-result} as  a “$\theta$ universal” expression for $\mathcal{R}_2$.

\begin{proof}
One has $a_2(a, P) = a_2(a, P_1) + a_2(a, P_2)$. Denote by $\mathcal{R}_2$ (resp.~$\mathcal{R}_{2}^{(1)}$, $\mathcal{R}_2^{(2)}$) the expressions associated to $P$ (resp.~$P_1$, $P_2$). Then one has $\mathcal{R}_2 = \mathcal{R}_2^{(1)} + \mathcal{R}_2^{(2)}$. 

The operator $P_1$ is a conformal like transformed Laplacian, so the computation of  $\mathcal{R}_2^{(1)}$ is a direct consequence of the results of Section~\ref{subsec-conf-like-lap}. Here the metric is constant, so that $R=0$, and it remains 
\begin{align*}
\mathcal{R}_2^{(1)}
& \cv \tfrac{1}{4\,\pi} \big[
F_{(1)\Delta k}(r_0, r_1) \, E_{r_0} (\Delta k) E_{r_1}
+ F_{(1)\partial k \partial k}^{\mu\nu}(r_0, r_1, r_2) \, E_{r_0} (\dmu k) E_{r_1} (\dnu k) E_{r_2}
\big],
\end{align*}
where, using \eqref{eq-conf-like-FDelta} and \eqref{eq-conf-like-Fdkdk},
\begin{align*}
F_{(1)\Delta k}(r_0, r_1)
&= 
-\sqrt{r_0} \, G_{q}(r_0, r_1)
- (\sqrt{r_0} + \sqrt{r_1}) \, G_{\hnabla \hnabla u}(r_0, r_1)
- (\sqrt{r_0} - \sqrt{r_1}) \, G_{\hnabla p}(r_0, r_1),
\\[1mm]
F_{(1)\partial k \partial k}^{\mu\nu}(r_0, r_1, r_2)
&= \begin{aligned}[t]
& 2 g^{\mu\nu} G_{\hnabla \hnabla u}(r_0, r_2)
\\
& + g^{\mu\nu} (\sqrt{r_0} + \sqrt{r_1})(\sqrt{r_1} + \sqrt{r_2}) \, G_{\hnabla u, \hnabla u}(r_0, r_1, r_2)
\\
& +  g^{\mu\nu} (\sqrt{r_0} - \sqrt{r_1})(\sqrt{r_1} + \sqrt{r_2}) \, G_{p, \hnabla u}(r_0, r_1, r_2)
\\
& + g^{\mu\nu} (\sqrt{r_0} + \sqrt{r_1})(\sqrt{r_1} - \sqrt{r_2}) \, G_{\hnabla u, p}(r_0, r_1, r_2)
\\
& + g^{\mu\nu} (\sqrt{r_0} - \sqrt{r_1})(\sqrt{r_1} - \sqrt{r_2}) \, G_{p, p}(r_0, r_1, r_2).
\end{aligned}
\end{align*}

For the operator $P_2$, one applies Theorem~\ref{thm-R-upq}:
\begin{align*}
\mathcal{R}_2^{(2)}
&= \tfrac{1}{4\,\pi} \big[
\begin{aligned}[t]
& (\sqrt{r_0} + \sqrt{r_1}) \, G_{\hnabla \hnabla u}(r_0, r_1) \, 
g^{\mu\nu} E_{r_0} (\dmu \dnu k) E_{r_1}
+ 2 G_{\hnabla \hnabla u}(r_0, r_2) \, 
g^{\mu\nu} E_{r_0} (\dmu k) E_{r_1} (\dnu k) E_{r_2}
\\
& + (\sqrt{r_0} + \sqrt{r_1}) (\sqrt{r_1} + \sqrt{r_2}) \, G_{\hnabla u, \hnabla u}(r_0, r_1, r_2) \, 
g^{\mu\nu} E_{r_0} (\dmu k) E_{r_1} (\dnu k) E_{r_2}
\\
& + (\sqrt{r_0} + \sqrt{r_1}) (\sqrt{r_1} + \sqrt{r_2}) \, G_{p, \hnabla u}(r_0, r_1, r_2) \, 
(h^{\nu\mu} - g^{\mu\nu}) E_{r_0} (\dmu k) E_{r_1} (\dnu k) E_{r_2}
\\
& + (\sqrt{r_0} + \sqrt{r_1}) (\sqrt{r_1} + \sqrt{r_2}) \, G_{\hnabla u, p}(r_0, r_1, r_2) \, 
(h^{\mu\nu} - g^{\mu\nu}) E_{r_0} (\dmu k) E_{r_1} (\dnu k) E_{r_2}
\\
& - (\sqrt{r_0} + \sqrt{r_1}) (\sqrt{r_1} + \sqrt{r_2}) \, G_{p, p}(r_0, r_1, r_2) \, 
g^{\mu\nu} E_{r_0} (\dmu k) E_{r_1} (\dnu k) E_{r_2}
\big]
\end{aligned}
\\
& \cv \tfrac{1}{4\,\pi} \,[
F_{(2)\Delta k}(r_0, r_1) \, \hlterm{E_{r_0} (\Delta k) E_{r_1}}
+ F_{(2)\partial k \partial k}^{\mu\nu}(r_0, r_1, r_2) \, \hlterm{E_{r_0} (\dmu k) E_{r_1} (\dnu k) E_{r_2}}
\,],
\end{align*}
with
\begin{align*}
F_{(2)\Delta k}(r_0, r_1)
&= 
- (\sqrt{r_0} + \sqrt{r_1}) \, G_{\hnabla \hnabla u}(r_0, r_1),
\\[1mm]
F_{(2)\partial k \partial k}^{\mu\nu}(r_0, r_1, r_2)
&= \begin{aligned}[t]
& 2 g^{\mu\nu} G_{\hnabla \hnabla u}(r_0, r_2)
\\
& + g^{\mu\nu} (\sqrt{r_0} + \sqrt{r_1})(\sqrt{r_1} + \sqrt{r_2}) \, G_{\hnabla u, \hnabla u}(r_0, r_1, r_2)
\\
& +  (h^{\nu\mu} - g^{\mu\nu}) (\sqrt{r_0} + \sqrt{r_1})(\sqrt{r_1} + \sqrt{r_2}) \, G_{p, \hnabla u}(r_0, r_1, r_2)
\\
& + (h^{\mu\nu} - g^{\mu\nu}) (\sqrt{r_0} + \sqrt{r_1})(\sqrt{r_1} + \sqrt{r_2}) \, G_{\hnabla u, p}(r_0, r_1, r_2)
\\
& - g^{\mu\nu} (\sqrt{r_0} + \sqrt{r_1})(\sqrt{r_1} + \sqrt{r_2}) \, G_{p, p}(r_0, r_1, r_2).
\end{aligned}
\end{align*}

Then $F_{\Delta k} \vc F_{1, \Delta k} + F_{2, \Delta k}$ and $F_{\partial k \partial k}^{\mu\nu} \vc F_{1, \partial k \partial k}^{\mu\nu} + F_{2, \partial k \partial k}^{\mu\nu}$ simplifies as in \eqref{eq-NCT2-FDk} and \eqref{eq-NCT2-Fdkdk}. The expression obtained for $\mathcal{R}_2$ shows that it belongs to $C(\TT^{2}_\Theta)$ and acts by left multiplication on $\H$.
\end{proof}

\subsection{The noncommutative four torus}
\label{4tore}

Our result applies to the computation of the conformally perturbed scalar curvature on the noncommutative four torus, computed in \cite{FathKhal13a, Fath15a}. In order to do that, as in dimension $2$, we perform the computation at rational value of $\theta$ as described at the end of Appendix \ref{sec-geometry-NCT}.

The operator we consider is the one in \cite{FathKhal13a}, written as
\begin{align*}
\Delta_\varphi 
&\vc
k^2 \bpartial_1 k^{-2} \partial_1 k^2 
+ k^2 \partial_1 k^{-2} \bpartial_1 k^2 
+ k^2 \bpartial_2 k^{-2} \partial_2 k^2 
+ k^2 \partial_2 k^{-2} \bpartial_2 k^2 
\end{align*}
with (in our notations) $k^2 \vc e^{h}$, $\partial_1 \vc - \delta_1 + i \delta_3$, $\bpartial_1 \vc -\delta_1 - i \delta_3$, $\partial_2 \vc -\delta_2 + i \delta_4$, and $\bpartial_2 \vc -\delta_2 - i \delta_4$. Indeed, in \cite{FathKhal13a, Fath15a}, the derivations are $\hdmu = -i \dmu$. This leads to
\begin{align*}
\Delta_\varphi 
&= - 2\left[ k^2 g^{\mu\nu} \dmu \dnu + g^{\mu\nu}  (\dnu k^2) \dmu + g^{\mu\nu} (\dmu \dnu k^2) - g^{\mu\nu}  (\dmu k^2) k^{-2} (\dnu k^2) \right]
\\
& \cv 2 P.
\end{align*}
The metric $g^{\mu\nu}$ is the diagonal one in \cite{FathKhal13a}, but in the following computation, we only require $g^{\mu\nu}$ to be constant.

Let us mention that in \cite{FathKhal13a, Fath15a}, the computation of the scalar curvature is done using $P$ defined above (and not  $\Delta_\varphi$), since the symbol in \cite[Lemma~3.6]{FathKhal13a} is the one of $P$. So we will use $P$ in the following. We get $P \cv - u g^{\mu\nu} \dmu \dnu - v^\mu \dmu - w$ with
\begin{align*}
u = k^2,
\qquad
v^\mu = g^{\mu\nu} (\dnu k^2),
\qquad
w = g^{\mu\nu} (\dmu \dnu k^2) - g^{\mu\nu}  (\dmu k^2) k^{-2} (\dnu k^2).
\end{align*}
Since $g$ is constant, we have as before $\alpha = R = \alpha^\mu = \beta^\mu = 0$ and $\hnmu = \nmu = \dmu$ and this implies $p^\mu = 0$ and $q = g^{\mu\nu} (\dmu \dnu k^2) - g^{\mu\nu}  (\dmu k^2) k^{-2} (\dnu k^2)$ in the covariant form \eqref{eq-def-P-upq}. We then use the result of Corollary~\ref{cor-d=4} to get the conformally perturbed scalar curvature:
\begin{align}
\label{eq-curv-NCT-4}
\mathcal{R}_2 
= \tfrac{1}{2^5 \pi^2} [
g^{\mu\nu} k^{-2} (\dmu \dnu k^2) k^{-2}  
- \tfrac{3}{2} g^{\mu\nu} k^{-2} (\dmu k^2) k^{-2} (\dnu k^2) k^{-2}
].
\end{align}

In Appendix~\ref{sec-comparison-NCT} it is shown that we recover the result previouly obtained in \cite{FathKhal13a, Fath15a} for the irrational noncommutative four torus.

\section{Conclusion}

In this paper, we have computed in all dimensions the local section $\mathcal{R}_2$ of $\End(V)$ defined by $a_2(a, P) = \int_M  \tr[a(x) \mathcal{R}_2(x)] \, \dvolg(x)$ for any section $a$ of $\End(V)$ for any nonminimal Laplace type operator $P = - [g^{\mu\nu} u(x)\pmu\pnu + v^\nu(x)\pnu + w(x)]$ (Theorems~\ref{thm-R-uvw} and \ref{thm-R-upq}). Expressions have been given for $\mathcal{R}_2$ in small dimensions, $d=2,3,4$ (Corollaries~\ref{cor-R-d=2},  \ref{cor-d=4}, and \ref{cor-R-d=3}) and for any even dimension $d \geq 2$ (Corollary~\ref{cor-R-d>2-even}), where, as expected from the results in \cite{IochMass17a}, polynomials expressions can be proposed. 

Despite the difficulties, $a_4(a,P)$ has been exhibited for $d=2$ in \cite{ConnesFath16} for the 2-dimensional noncommutative torus, leaving open the computation of $\mathcal{R}_4$. We hope that our method could be used to reach $\mathcal{R}_4$ in any dimension, using a computer algebra system in the more general framework of an arbitrary $P$, like~\eqref{eq-def-P-uvw}.

Our method still applies to more general setting than the NCT at rational values of the deformation parameter, namely to $n$-homogeneous $C^*$-algebras, which can be characterized in terms of sections of fiber bundles with fiber space $M_n(\C)$ \cite{Fell61a, Blac06a}.

\section*{Acknowledgements}

The authors are indebted to Valentin Zagrebnov for helpful discussions concerning some aspects of this paper.

\appendix
\renewcommand*{\thesection}{\Alph{section}}

\section{Geometrical identification of the noncommutative torus at rational values}
\label{sec-geometry-NCT}

For $d=2$ and $\theta \vc p/q$ rational, with $p,q$ relatively prime integers with $q>0$, it is known (see \cite[Prop.~12.2]{GracVariFigu01a}, \cite[Sect.~3]{DuboKrieMaedMich02a}) that the algebra $C(\TT^{2}_\Theta)$ of the NCT identifies with the algebra $\Gamma(A_{\theta})$ of continuous sections of a fiber bundle $A_{\theta}$ in $M_q(\C)$ algebras over a $2$-torus $\TT^{2}_B$. Let us describe this identification.

Denote by $\TT^{2}_P$ the $2$-torus given by identification of opposite sides of the square $[0, 2\pi]^2$. An element in $\TT^{2}_P$ is written as $(e^{i x}, e^{i y})$ for $(x,y) \in [0, 2\pi]^2$. There is a natural action of the (abelian discrete) group $G \vc \Z_q^2$ on $\TT^{2}_P$: $(m,n) \cdot (e^{i x}, e^{i y}) \vc (e^{i (x + 2 \pi p m/q)}, e^{i (y + 2 \pi p n/q)})$. 
\\The quotient $\TT^{2}_B \vc \TT^{2}_P/G$ is the $2$-torus constructed by identification of opposite sides of the square $[0, 2\pi/q]^2$. Indeed, there are unique $m \in \Z_q$ and $n \in \Z_q$ such that $e^{i (x + 2 \pi p m/q)} = e^{i (x + 2 \pi /q)}$ and $e^{i (y + 2 \pi p n/q)} = e^{i (y + 2 \pi /q)}$, so that $(m,0)$ (resp. $(0,n)$) identifies $(e^{i x}, e^{i y})$ with $(e^{i (x + 2 \pi /q)}, e^{i y})$ (resp. $(e^{i x}, e^{i y})$ with $(e^{i x}, e^{i (y + 2 \pi /q)})$) in $\TT^{2}_P/G$. The quotient map $\TT^{2}_P \to \TT^{2}_B$ is a $G$-covering. 

Let us now consider the $C^\ast$-algebra $C(\TT^{2}_P, M_q(\C)) \simeq C(\TT^{2}_P) \otimes M_q(\C)$ of matrix-valued continuous functions on $\TT^{2}_P$, in which the space of smooth functions $C^\infty(\TT^{2}_P, M_q(\C))$ is a dense subalgebra. In order to describe this algebra, let consider the two matrices
\begin{align*}
U_0 &\vc 
\begin{pmatrix}
0 & 1 & 0 & \cdots & 0 \\
0 & 0 & 1 & \cdots & 0 \\
\vdots & & & \ddots & \vdots \\
0 & 0 & 0 & \cdots & 1 \\
1 & 0 & 0 & \cdots & 0
\end{pmatrix},
&
V_0 & \vc
\begin{pmatrix}
1 & 0 & 0 & \cdots & 0 \\
0 & \xi_1 & 0 & \cdots & 0 \\
\vdots & & \ddots & & \vdots \\
0 & \cdots & 0 & \xi_{q-2} & 0 \\
0 & 0 & \cdots  & 0 & \xi_{q-1}
\end{pmatrix}
\quad \text{with } \xi_n\vc e^{i 2 \pi n\theta},
\end{align*}
which satisfy $U_0 V_0 = e^{i 2 \pi \theta} V_0 U_0$, $U_0^q = V_0^q = \bbbone_q$. For $(r,s) \in \Z_q^2$, the $U_0^r V_0^s$'s define a basis of $M_q(\C)$ such that $\tr[ U_0^r V_0^s ] = q \,\delta_{(r,s),(0,0)}$ (here $\tr$ is the trace on $M_q(\C)$). Then $a \in C^\infty(\TT^{2}_P, M_q(\C))$ can be decomposed as
\begin{equation}
\label{eq-dec-smooth-torus-matrix}
a(e^{i x}, e^{i y}) 
= \sum_{(r,s) \in \Z_q^2} a_{r,s}\,(e^{i x}, e^{i y}) \,U_0^r V_0^s
= \sum_{(k,\ell) \in \Z^2} \sum_{(r,s) \in \Z_q^2} a_{k,\ell, r, s} \,u(x)^{k} v(x)^\ell\, U_0^r V_0^s
\end{equation}
where, with $u(x) \vc e^{i x}$ and $v(x) \vc e^{i y}$, the last decomposition is the Fourier series of the smooth functions  $a_{r,s}$ on $\TT^{2}_P$. In particular, $a_{k,\ell, r, s}$ are rapidly decreasing coefficients in terms of $(k,\ell) \in \Z^2$.

The group $G$ acts on $M_q(\C)$ by
\begin{equation*}
(m,n)\cdot A \vc U_0^{-n} V_0^{m} A V_0^{-m} U_0^{n}.
\end{equation*}
Let us consider the subalgebra $C_G(\TT^{2}_P, M_q(\C)) \subset C(\TT^{2}_P, M_q(\C))$ of $G$-equivariant functions, which by definition satisfy, for any $(m,n) \in G$, 
\begin{equation*}
a(e^{i (x + 2 \pi p m/q)}, e^{i (y + 2 \pi p n/q)}) = U_0^{n} V_0^{-m} a(e^{i x}, e^{i y}) V_0^{m} U_0^{-n}.
\end{equation*}
Using \eqref{eq-dec-smooth-torus-matrix}, the $G$-equivariant elements in $C^\infty(\TT^{2}_P, M_q(\C))$ are such that their coefficients satisfy $a_{k,\ell, r, s} \,e^{i 2 \pi(m k + n \ell)} = a_{k,\ell, r, s}\, e^{i 2 \pi(m r + n s)}$ for any $(m,n) \in G$, $(k,\ell) \in \Z^2$ and $(r,s) \in \Z_q^2$, so that $a_{k,\ell, r, s} \neq 0$ only when $m k + n \ell \equiv m r + n s \mod q$. With $(m,n) = (1,0)$ and $(0,1)$ this implies $k \equiv r \mod q$ and $\ell \equiv s \mod q$. In \eqref{eq-dec-smooth-torus-matrix}, for a couple $(k,\ell) \in \Z^2$, there is a unique $(r,s) \in \Z_q^2$ for which $a_{k,\ell, r, s} \neq 0$ ($r$ and $s$ are the remainders of the Euclidean divisions of $k$  and $\ell$ by $q$). Then, the only non zero coefficients $a_{k,\ell, r, s}$ depends only on $(k,\ell) \in \Z^2$. We denote them by $a_{k,\ell}$, and a smooth $G$-equivariant function $a \in C^\infty_G(\TT^{2}_P, M_q(\C))$ is then given by the expansion
\begin{equation*}
a = \sum_{(k,\ell) \in \Z^2}  a_{k,\ell} \,(u U_0)^{k} (v V_0)^\ell = \sum_{(k,\ell) \in \Z^2}  a_{k,\ell}\, U^{k} V^\ell
\end{equation*}
with $U \vc u U_0$, $V \vc v V_0$ satisfying $U V = e^{i 2 \pi \theta} V U$. Then, the $C^\ast$-algebra $C(\TT^{2}_\Theta)$ for $\theta = p/q$ identifies with $C_G(\TT^{2}_P, M_q(\C))$, the $C^\ast$-completion of $C^\infty_G(\TT^{2}_P, M_q(\C))$ in $C(\TT^{2}_P, M_q(\C))$.

The space $C_G(\TT^{2}_P, M_q(\C))$ identifies in a canonical way with the space $\Gamma(A_\theta)$ of continuous sections of the associated fiber bundle $A_\theta \vc \TT^{2}_P \times_G M_q(\C)$ to the $G$-covering $\TT^{2}_P \to \TT^{2}_B$. By definition, $A_\theta$ is the quotient of $\TT^{2}_P \times M_q(\C)$ by the equivalence relation $((m,n)\cdot (e^{i x}, e^{i y}), A) \sim ((e^{i x}, e^{i y}), (m,n)\cdot A)$ for any $(m,n) \in G$. We denote by $[(e^{i x}, e^{i y}), A] \in A_\theta$ the class of $((e^{i x}, e^{i y}), A)$. To $a \in C_G(\TT^{2}_P, M_q(\C))$ corresponds the section $s \in \Gamma(A_\theta)$ defined by $s(x,y) \vc [(e^{i x}, e^{i y}), a(e^{i x}, e^{i y})]$. 

In the  GNS construction, $C^\infty_G(\TT^{2}_P, M_q(\C))$ is dense in $\H$ and is contained in the domains of the $\dmu$'s. The fiber of the vector bundle $V$ on which the differential operator $P$ acts is then $\C^{N} \simeq M_q(\C)$, \textit{i.e.} $N=q^2$. In the present situation, all used elements in $\Gamma(\End(V))$ are in fact left multiplications by elements in $C_G(\TT^{2}_P, M_q(\C)) \simeq \Gamma(A_\theta)$. For instance, the element $a$ in \eqref{eq-Rr-trace-choice} will be understood as the left multiplication by an element $a \in C_G(\TT^{2}_P, M_q(\C))$.

For $A \in M_q(\C)$, let $L(A)$ be the left multiplication by $A$ on $M_q(\C)$. Then $L(A)$ has the same spectrum as $A$, each eigenvalues having $q$ times its original multiplicity. In particular, we have that
\begin{equation*}
\tr[ L(A) ] = q \tr[ A ]
\end{equation*}
where in the LHS $\tr$ is the trace of operators on $M_q(\C)$ and in the RHS $\tr$ is the trace on $M_q(\C)$.

The computation of $\mathcal{R}_2$ uses local trivializations of sections of $A_\theta$. To $s \in \Gamma(A_\theta)$, we associate the local section $s_\text{loc} : (0, 2\pi/q)^2 \to M_q(\C)$ defined by $s_\text{loc}(x,y) \vc a(e^{i x}, e^{i y})$. Notice that the open subset $(0, 2\pi/q)^2 \subset \TT^{2}_B$ is sufficient to describe the continuous section $s$ via its trivialization $s_\text{loc}$. The (local) section $\mathcal{R}_2$ relative to $\varphi$ in \eqref{eq-Rr-trace-choice} is defined by
\begin{equation*}
\varphi(L(s_\text{loc})) = \int_{\TT^{2}_B} \tr[L(s_\text{loc}(x,y))] \, \dvolg(x,y)
= \abs{g}^{1/2} \int_0^{2\pi/q} \dd x \int_0^{2\pi/q} \dd y \, q \tr[ s_\text{loc}(x,y) ].
\end{equation*}
where we suppose here that $\abs{g}^{1/2} $ is constant (this is the case for the situations considered in the paper). For $s$ associated to $a \in C^\infty_G(\TT^{2}_P, M_q(\C))$, one has 
\begin{equation*}
s_\text{loc}(x,y) = \sum_{(k,\ell) \in \Z^2}  a_{k,\ell} \,(u(x) U_0)^{k}\, (v(x) V_0)^\ell, 
\end{equation*}
hence
\begin{equation*}
\varphi(L(s_\text{loc})) 
= \abs{g}^{1/2} \, q \sum_{(k,\ell) \in \Z^2}  a_{k,\ell} \int_0^{2\pi/q} \dd x \, e^{i k x} \int_0^{2\pi/q} \dd y  \, e^{i \ell y}  \,  \tr[ U_0^{k} V_0^\ell ].
\end{equation*}
The trace $\tr[ U_0^{k} V_0^\ell ]$ is non-zero only when $k,\ell$ are multiple of $q$, and its value is then $q$, so that
\begin{equation}
\label{eq-varphi-NCt-computation}
\varphi(L(s_\text{loc})) 
= \abs{g}^{1/2} \, q^2\! \!\!\!\sum_{(k,\ell) \in \Z^2}\!  a_{k,\ell} \int_0^{2\pi/q} \!\!\dd x \, e^{i q k x} \int_0^{2\pi/q}\! \!\dd y  \, e^{i q \ell y}
=  \abs{g}^{1/2} \, q^2 \big(\tfrac{2\pi}{q}\big)^2 a_{0,0}
= (2\pi)^2 \abs{g}^{1/2}  \, \NCtr(a).
\end{equation}
Finally we get (when $ \abs{g}^{1/2} $ is constant)
\begin{equation*}
\varphi \circ L = (2\pi)^2 \abs{g}^{1/2}  \, \NCtr
\end{equation*}
when applied to any elements in $C(\TT^2_\Theta)$.

\bigskip
Consider now a $4$-dimensional noncommutative torus for $\Theta = \left(\begin{smallmatrix} \theta_1 \chi & 0 \\ 0 & \theta_2 \chi \end{smallmatrix}\right)$, as in \eqref{def: matrice theta}, and $\theta_i = p_i/q_i$, $p_i,q_i$ relatively prime integers, and  $q_i>0$. \\
Then $C(\TT^4_\Theta) = C(\TT^2_{\Theta_1}) \otimes C(\TT^2_{\Theta_2}) = \Gamma(A_{\theta_1}) \otimes \Gamma(A_{\theta_2}) = \Gamma(A_{\theta_1} \boxtimes A_{\theta_2})$ where $A_{\theta_1} \boxtimes A_{\theta_2}$ is the external tensor product of the two vector bundles $A_{\theta_i}$ over the base $2$-torus $\TT^{2}_{B, i}$ defined as above. Recall that, with $\text{pr}_i : \TT^{2}_{B, 1} \times \TT^{2}_{B, 2} \to \TT^{2}_{B, i}$ the natural projections, $A_{\theta_1} \boxtimes A_{\theta_2} \vc \left( \text{pr}_1^\ast A_{\theta_1} \right) \otimes \left( \text{pr}_2^\ast A_{\theta_2} \right)$ where $\text{pr}_i^\ast A_{\theta_i}$ is the pull-back of $A_{\theta_i}$ on $\TT^{2}_{B, 1} \times \TT^{2}_{B, 2}$. The fiber of $A_{\theta_1} \boxtimes A_{\theta_2}$ is then $M_{q_1}(\C) \otimes M_{q_2}(\C) \simeq M_{q_1 q_2}(\C)$ and the isomorphism $\Gamma(A_{\theta_1}) \otimes \Gamma(A_{\theta_2}) \xrightarrow{\simeq} \Gamma(A_{\theta_1} \boxtimes A_{\theta_2})$ is induced by $(s_1 \otimes s_2)(x_1, x_2) \vc s_1(x_1) \otimes s_2(x_2)$ for any $s_i \in \Gamma(A_{\theta_i})$ and $x_i \in \TT^{2}_{B, i}$. Using the same line of arguments as for the $2$-dimensional case, and denoting by $g_i$ a constant metric on $\TT^{2}_{B, i}$, one gets, for any $a_i \in  C(\TT^2_{\Theta_i})$,
\begin{equation}
\label{eq-varphi-NCt-computation-d=4}
\varphi\circ L(a_1 \otimes a_2) =\abs{g_1}^{1/2}  \abs{g_2}^{1/2} \, (q_1 q_2)^2 \left( \tfrac{2 \pi}{q_1}\right)^2 \left( \tfrac{2 \pi}{q_2}\right)^2 \NCtr\,(a_1 \otimes a_2),
\end{equation}
so that
\begin{equation}
\label{eq-varphi-NCtr-d=4}
\varphi \circ L = (2\pi)^4 \abs{g_1}^{1/2}  \abs{g_2}^{1/2} \, \NCtr\,.
\end{equation}

This procedure can be extended straightforwardly to any even dimension.

\section{Comparison with previous results for noncommutative tori}
\label{sec-comparison-NCT}

We would like to compare the result \eqref{eq-NCT-result} with \cite[Theorem~5.2]{FathKhal11a}. Some transformations are in order, since some conventions are different and the results are presented using different operators. In \cite[Theorem~5.2]{FathKhal11a}, it is presented relative to the normalized trace $\NCtr$ on $C(\TT^2_\Theta)$, while our result is presented relative to $\varphi \circ L = (2\pi)^2 \ \abs{g}^{1/2} \, \NCtr$ with here $ \abs{g}^{1/2} = \tau_2^{-1}$. If $\mathcal{R}_{FK}$ denotes the operator of \cite[Theorem~5.2]{FathKhal11a}, then
\begin{align*}
a_2(a, P)
 &= \NCtr\,( a \mathcal{R}_{FK}) = \varphi(a \mathcal{R}_2)
= (2\pi)^2 \tau_2^{-1} \, \NCtr\,(a \mathcal{R}_2),\quad \forall a \in C(\TT^2_\Theta),
\end{align*}
so we need to show that $\mathcal{R}_2 = \tfrac{\tau_2}{(2 \pi)^2} \mathcal{R}_{FK}$ (strictly speaking, $a \mathcal{R}_2$ should be replaced by $L(s \mathcal{R}_2)$ in $\varphi(a \mathcal{R}_2)$, see Section~\ref{sec-geometry-NCT}).

Present results are given using functional calculus on the left and right multiplication operators $L_u$ and $R_u$ where $u=k^2$. The corresponding spectral decompositions give $L_u(a) = \sum_{r_0} r_0 \, E_{r_0} a$ and $R_u(a) = \sum_{r_1} r_1 \, a E_{r_1}$, where $E_{r_i}$ is the projection associated to $u$ for the spectral value $r_i$. In \cite{FathKhal11a}, another convention is used, namely via functional calculus on the modular operator $\Modular(a) \vc k^{-2} a k^2$. If $E_y^\Modular$ denotes the projection of $\Modular$ associated to the spectral value $y$, then
\begin{align*}
\Modular(a) &= L_{u^{-1}} \circ R_u (a)
= \sum_{r_0, r_1} r_0^{-1} r_1 \, E_{r_0} a E_{r_1}
= \sum_{r_0, y} y \, E_{r_0} a E_{y r_0}
= \sum_{y} y \, E_y^\Modular(a)
\end{align*}
where $y \vc r_0^{-1} r_1$ belongs to the spectrum of $\Modular$.\\
Thus $E_y^\Modular(a) = \sum_{r_0} E_{r_0} a E_{y r_0}$ and is easy to check that $E_{y_0}^\Modular E_{y_1}^\Modular = \delta_{y_0, y_1} E_{y_0}^\Modular$.

The next three lemmas are technical results which permit to transform our relation \eqref{eq-NCT-result} into a relation that will be compared to \cite[Theorem~5.2]{FathKhal11a}. 

For any $b_0 \otimes \cdots \otimes b_p \in M_N(\C)^{\otimes^{p+1}}$, we denote by $\Modular_i$ the operator $\Modular$ which acts on $b_i$, by $L^i_k$ the left multiplication by $k$ acting on $b_i$, and by $R^i_u$ the right multiplication by $u$ on $b_i$. Notice that all these operators commute.\\
The first lemma transforms the functional calculus in the $R^i_u$'s into a functional calculus in the $\Modular_i$'s.

\begin{lemma}[Rearrangement Lemma]
\label{lem-rearrangement-lemma}
For any function $F(r_0, r_1, \dots, r_p)$ of the eigenvalues $r_i$ of the $R^i_u$'s and any $b_0 \otimes \cdots \otimes b_p$ one has
\begin{align*}
\sum_{r_0, r_1, \dots, r_p} F(r_0, r_1, \dots, r_p) \, b_0 E_{r_0} b_1 E_{r_1} \cdots b_p E_{r_p}
=
\sum_{r_0, y_1, \dots, y_p} f(r_0, y_1, \dots, y_p) \, b_0 E_{r_0} E_{y_1}^\Modular(b_1) \cdots E_{y_p}^\Modular(b_p)
\end{align*}
where $f(r_0, y_1, \dots, y_p) \vc F(r_0, r_0 y_1, r_0 y_1 y_2, \dots, r_0 y_1 \cdots y_p)$ is a spectral function of $R^0_u$ and the $\Modular_i$'s.
\end{lemma}

Using functional calculus notation, this lemma implies
\begin{align*}
\mul \circ F(R^0_u, R^1_u, \dots, R^p_u) = \mul \circ f(R^0_u,  \Modular_1, \dots, \Modular_p)
\end{align*}
as operators acting on elements $b_0 \otimes \cdots \otimes b_p$. This result is very analog to the rearrangement lemma \cite[Corollary~3.9]{Lesc14a} without the integral $\int_0^\infty du$ in \cite[eq.~(3.9)]{Lesc14a}.

\begin{proof}
It is sufficient to show how the combinatorial aspect of the proof works for $p=2$. One has
\begin{align*}
\sum_{r_0, y_1, y_2} f(r_0, y_1, y_2) \, b_0 E_{r_0} E_{y_1}^\Modular(b_1) E_{y_2}^\Modular(b_2)
&=
\sum_{r_0, y_1, y_2} F(r_0, r_0 y_1, r_0 y_1 y_2) \, b_0 E_{r_0} E_{y_1}^\Modular(b_1) E_{y_2}^\Modular(b_2)
\\
&= 
\sum_{\substack{r_0, y_1, y_2\\ z_1, z_2}} F(r_0, r_0 y_1, r_0 y_1 y_2) \, b_0 E_{r_0} E_{z_1} b_1 E_{y_1 z_1} E_{z_2} b_2 E_{y_2 z_2}
\\
&= 
\sum_{\substack{r_0, y_1, y_2\\ z_2}} F(r_0, r_0 y_1, r_0 y_1 y_2) \, b_0 E_{r_0} b_1 E_{y_1 r_0} E_{z_2} b_2 E_{y_2 z_2}
\\
&= 
\sum_{\substack{r_0, r_1, y_2\\ z_2}} F(r_0, r_1, r_1 y_2) \, b_0 E_{r_0} b_1 E_{r_1} E_{z_2} b_2 E_{y_2 z_2}
\\
&= 
\sum_{\substack{r_0, r_1, y_2}} F(r_0, r_1, r_1 y_2) \, b_0 E_{r_0} b_1 E_{r_1} b_2 E_{y_2 r_1}
\\
&= 
\sum_{\substack{r_0, r_1, r_2}} F(r_0, r_1, r_2) \, b_0 E_{r_0} b_1 E_{r_1} b_2 E_{r_2}.
\end{align*}
\end{proof}

Let $k = e^{h/2}$. While the arguments $b_i$ mentioned above are $\dmu k$ or $\Delta k$, they are $\dmu (\ln k) = \tfrac{1}{2}\dmu h$ and $\Delta (\ln k) = \tfrac{1}{2}\Delta h$ in \cite{FathKhal11a}. The second lemma gives the relations between these arguments.

\begin{lemma}
\label{lem-k-to-logk}
If $
g_1(y) = \frac{\sqrt{y} - 1}{\ln y}
\text{ and }
g_2(y_1, y_2) =2 \frac{ \sqrt{y_1}(\sqrt{y_2} - 1) \ln y_1 - (\sqrt{y_1} - 1) \ln y_2}{ \ln y_1 \ln y_2 \,(\ln y_1 + \ln y_2)},
$
then
\begin{align*}
\dmu k &= 
k \, g_1(\Modular)[ \dmu h ] = 2 k \, g_1(\Modular)[ \dmu \ln k ],
\\
\Delta k &=
k \, g_1(\Modular)[ \Delta h ] 
- g^{\mu\nu} k \, \mul \circ g_2(\Modular_1, \Modular_2) [ (\dmu h) \otimes (\dnu h) ]
\\
&=
2 k \, g_1(\Modular)[ \Delta \ln k ] 
- 4 g^{\mu\nu} k \, \mul \circ g_2(\Modular_1, \Modular_2) [ (\dmu \ln k) \otimes (\dnu \ln k)].
\end{align*}
\end{lemma}

\begin{proof}
With $g_1(y) \vc \tfrac{1}{2} \int_0^1 d s_1 \,y^{s_1/2} =( \sqrt{y} - 1)\ln^{-1} y$, we get
\begin{align*}
\dmu k &= 
\dmu e^{h/2}
= \int_0^1 d s_1\, e^{(1-s_1)h/2} \,(\dmu h/2)\, e^{s_1 h/2}
= \tfrac{1}{2} k ( \int_0^1 d s_1 \,\Modular^{s_1/2} ) [ \dmu h ]
\\
&= k \frac{\Modular^{1/2} - 1}{\ln \Modular} [ \dmu h ]
= 2 k \frac{\Modular^{1/2} - 1}{\ln \Modular} [ \dmu \ln k ]
= k \, g_1(\Modular)[ \dmu h ] = 2 k \, g_1(\Modular)[ \dmu \ln k ].
\end{align*}
Similarly for the Laplacian,
\begin{align*}
\Delta k 
&= - g^{\mu\nu} (\dmu \dnu k)
= - \tfrac{1}{2} g^{\mu\nu} \dmu [ \int_0^1 d s_1 e^{(1-s_1)h/2} (\dnu h) e^{s_1 h/2} ]
\\
&= \begin{aligned}[t]
&- \tfrac{1}{2} g^{\mu\nu} \,[ \int_0^1 d s_1\, e^{(1-s_1)h/2} (\dmu \dnu h) e^{s_1 h/2} \,]
\\
&- \tfrac{1}{4} g^{\mu\nu} \,[ \int_0^1 d s_1 \int_0^{1-s_1} \,d s_2  \,e^{(1 -s_1 - s_2)h/2} (\dmu h) e^{s_2 h/2} (\dnu h) e^{s_1 h/2} \,]
\\
&- \tfrac{1}{4} g^{\mu\nu} \,[ \int_0^1 d s_1 \int_0^{s_1} d s_2 \, e^{(1 -s_1)h/2} (\dnu h) e^{(s_1 - s_2)h/2} \dmu h) e^{s_2 h/2} \,]
\end{aligned}
\\
&= \begin{aligned}[t]
&- \tfrac{1}{2} g^{\mu\nu} k [ \int_0^1 d s_1 \, \Modular^{s_1/2} ] (\dmu \dnu h) 
\\
&- \tfrac{1}{4} g^{\mu\nu} k \, \mul \circ [ \int_0^1 d s_1 \int_0^{1-s_1} d s_2 \,  \Modular_1^{(s_1 + s_2)/2} \Modular_2^{s_1/2} ] \,[ (\dmu h) \otimes (\dnu h) ]
\\
&- \tfrac{1}{4} g^{\mu\nu} k \, \mul \circ \,[ \int_0^1 d s_1 \int_0^{s_1} d s_2 \,  \Modular_1^{s_1/2} \Modular_2^{s_2/2} \,] \,[ (\dmu h) \otimes (\dnu h) ]
\end{aligned}
\\
&= 
k \, g_1(\Modular)[ \Delta h ] 
- g^{\mu\nu} k \, \mul \circ g_2(\Modular_1, \Modular_2) [ (\dmu h) \otimes (\dnu h) ]
\end{align*}
with
\begin{align*}
g_2(y_1, y_2) &\vc 
\tfrac{1}{4} \int_0^1 d s_1 \int_0^{1-s_1} d s_2 \,  y_1^{(s_1 + s_2)/2} \,y_2^{s_1/2}
+ \tfrac{1}{4} \int_0^1 d s_1 \int_0^{s_1} d s_2 \,  y_1^{s_1/2} \,y_2^{s_2/2} 
\\
&=
2 \frac{ \sqrt{y_1}(\sqrt{y_2} - 1) \ln y_1 - (\sqrt{y_1} - 1) \ln y_2}{ \ln y_1 \ln y_2 (\ln y_1 + \ln y_2)}\,.
\end{align*}
\end{proof}

The third lemma gives (technical) functional relations which allow a change of arguments inside our operators. Denote by $\mul_{12}[b_0 \otimes b_1 \otimes b_2] \vc b_0 \otimes b_1 b_2$ the partial multiplication. 

\begin{lemma}
\label{lem-composition-change-variables}
For any operators like $f_1(R^0_u, \Modular_1)$, $f_2(R^0_u, \Modular_1, \Modular_2)$, $g_1(\Modular_1)$, and $g_2(\Modular_1, \Modular_2)$, one has
\begin{align*}
\mul \circ f_1(R^0_u, \Modular_1) \circ L^1_k \circ g_1(\Modular_1)
&=
\mul \circ R^0_k \circ f_1(R^0_u, \Modular_1) \circ g_1(\Modular_1),
\\
\mul \circ f_1(R^0_u, \Modular_1) \circ \mul_{12} \circ L^1_k \circ g_2(\Modular_1, \Modular_2) 
&=
\mul \circ R^0_k \circ f_1(R^0_u, \Modular_1 \Modular_2) \circ g_2(\Modular_1, \Modular_2) ,
\\
\mul \circ f_2(R^0_u, \Modular_1, \Modular_2) \circ L^1_k \circ L^2_k \circ g_2(\Modular_1, \Modular_2)
&=
\mul \circ R^0_u \circ \Modular_1^{1/2} \circ f_2(R^0_u, \Modular_1, \Modular_2) \circ g_2(\Modular_1, \Modular_2) .
\end{align*}
Thus the operators on the LHS are respectively associated, modulo the multiplication operator $\mul$, to operators defined by the spectral functions 
\begin{align*}
\sqrt{r_0} f_1(r_0, y_1) g_1(y_1),
\qquad
\sqrt{r_0} f_1(r_0, y_1 y_2) g_2(y_1, y_2),
\qquad r_0 \sqrt{y_1} f_2(r_0, y_1, y_2) g_2(y_1, y_2),
\end{align*}
where $y_1, y_2$ belong to the spectrum of $\Modular$ and $r_0$ to the spectrum of $u$.
\end{lemma}

\begin{proof}
For the first relation, we compute the LHS on $b_0 \otimes b_1$ using spectral decomposition:
\begin{align*}
\mul \circ f_1(R^0_u, \Modular_1) \circ L^1_k  \circ g_1(\Modular_1) [ b_0 \otimes b_1]
&= \sum_{\substack{r_0, r_1, \\ y, y_1}} f_1(r_0, y) \sqrt{r_1} g_1(y_1)  \, 
b_0 E_{r_0} E_y^\Modular [ E_{r_1} E_{y_1}^\Modular(b_1) ]
\\
&=
\sum_{\substack{r_0, r_1, \\ y, y_1, \\ z, z_1}} f_1(r_0, y) \sqrt{r_1} g_1(y_1)  \, 
b_0 E_{r_0} E_z E_{r_1} E_{z_1} b_1 E_{y_1 z_1} E_{y z}
\end{align*}
the projections products imply $r_0 = z = r_1 = z_1$ and $y_1 z_1 = y z$, so this is equal to
\begin{align*}
\sum_{\substack{r_0, y, \\ z}} f_1(r_0, y) \sqrt{r_0} g_1(y) \, 
b_0 E_{r_0} E_z b_1 E_{y z}
&=
\sum_{\substack{r_0, y, \\ z}} \sqrt{r_0} f_1(r_0, y) g_1(y) \, 
b_0 E_{r_0} E_y^\Modular(b_1)
\\
&=
\mul \circ R^0_k \circ f_1(R^0_u, \Modular_1) \circ g_1(\Modular_1) [ b_0 \otimes b_1 ].
\end{align*}
For the second relation, we compute the LHS on $b_0 \otimes b_1 \otimes b_2$:
\begin{align*}
\mul \circ f_1(R^0_u, \Modular_1) \circ \mul_{12}  & \circ L^1_k \circ g_2(\Modular_1, \Modular_2) [ b_0 \otimes b_1 \otimes b_2 ]
\\
&=
\sum_{\substack{r_0, r_1, \\ y, y_1, y_2}} f_1(r_0, y) \sqrt{r_1} g_2(y_1, y_2)  \, 
b_0 E_{r_0} E_y^\Modular [ E_{r_1} E_{y_1}^\Modular(b_1) E_{y_2}^\Modular(b_2) ]
\\
&=
\sum_{\substack{r_0, r_1, \\ y, y_1, y_2,\\ z, z_1, z_2}} f_1(r_0, y) \sqrt{r_1} g_2(y_1, y_2)  \, 
b_0 E_{r_0} E_z E_{r_1} E_{z_1} b_1 E_{y_1 z_1} E_{z_2} b_2 E_{y_2 z_2} E_{y z}
\\
\intertext{which implies $r_0 = z = r_1 = z_1$, $y_1 z_1 = z_2$, and $y_2 z_2 = y z$, and}
&=
\sum_{\substack{r_0, z_1, z_2, \\y_1, y_2}} f(r_0, y_1 y_2) \sqrt{r_0} g_2(y_1, y_2)  \, 
b_0 E_{r_0} E_{z_1} b_1 E_{y_1 z_1} E_{z_2} b_2 E_{y_2 z_2}
\\
&=
\sum_{r_0, y_1, y_2} \sqrt{r_0} f(r_0, y_1 y_2) g_2(y_1, y_2)  \, 
b_0 E_{r_0} E_{y_1}^\Modular(b_1) E_{y_2}^\Modular(b_2)
\\
&= 
\mul \circ R^0_k \circ f(R^0_u, \Modular_1 \Modular_2) \circ g_2(\Modular_1, \Modular_2) [ b_0 \otimes b_1 \otimes b_2 ].
\end{align*}
For the third relation, we compute the LHS on $b_0 \otimes b_1 \otimes b_2$:
\begin{align*}
\mul & \circ f(R^0_u, \Modular_1, \Modular_2) \circ L^1_k \circ L^2_k \circ g_2(\Modular_1, \Modular_2) [ b_0 \otimes b_1 \otimes b_2 ]
\\
&=
\sum_{\substack{r_0, r_1, r_2 ,\\ y_1, y_2, y'_1, y'_2}} f_2(r_0, y_1, y_2) \sqrt{r_1} \sqrt{r_2} g_2(y'_1, y'_2) \, 
b_0 E_{r_0} E_{y_1}^\Modular [ E_{r_1} E_{y'_1}^\Modular(b_1) ] E_{y_2}^\Modular [ E_{r_2} E_{y'_2}^\Modular(b_2) ]
\\
&=
\sum_{\substack{r_0, r_1, r_2 ,\\ y_1, y_2, y'_1, y'_2, \\ z_1, z_2, z'_1, z'_2}} f_2(r_0, y_1, y_2) \sqrt{r_1} \sqrt{r_2} g_2(y'_1, y'_2) \, 
b_0 E_{r_0} E_{z_1} E_{r_1} E_{z'_1} b_1 E_{y'_1 z'_1} E_{y_1 z_1} E_{z_2} E_{r_2} E_{z'_2} b_2 E_{y'_2 z'_2} E_{y_2 z_2}
\\
\intertext{which implies $r_0 = z_1 = r_1 = z'_1$, $y'_1 z'_1 = y_1 z_1 = z_2 = r_2 = z'_2$, and $y'_2 z'_2 = y_2 z_2$, so that:}
&=
\sum_{\substack{r_0,  z_1, z_2, \\ y_1, y_2}} f_2(r_0, y_1, y_2) \sqrt{r_0} \sqrt{y_1 r_0} g_2(y_1, y_2) \, 
b_0 E_{r_0} E_{z_1} b_1 E_{y_1 z_1} E_{z_2} b_2 E_{y_2 z_2}
\\
&=
\sum_{r_0,  y_1, y_2} r_0 \sqrt{y_1} f_2(r_0, y_1, y_2) g_2(y_1, y_2) \, 
b_0 E_{r_0} E_{y_1}^\Modular(b_1) E_{y_2}^\Modular(b_2)
\\
&=
\mul \circ R^0_u \circ \Modular_1^{1/2} \circ f_2(R^0_u, \Modular_1, \Modular_2) \circ g_2(\Modular_1, \Modular_2) [ b_0 \otimes b_1 \otimes b_2 ].
\end{align*}
\end{proof}

We can now change \eqref{eq-NCT-result} in order to compare with \cite[Theorem~5.2]{FathKhal11a}. As in Lemma~\ref{lem-rearrangement-lemma}, let 
\begin{align*}
f_{\Delta k}(r_0, y_1)
\vc
F_{\Delta k}(r_0, r_0 y_1),
\qquad
f_{\partial k \partial k}^{\mu\nu}(r_0, y_1, y_2)
\vc F_{\partial k \partial k}^{\mu\nu}(r_0, r_0 y_1, r_0 y_1 y_2).
\end{align*}
Using Lemmas~\ref{lem-k-to-logk} and \ref{lem-composition-change-variables}, one gets
\begin{align*}
\mul & \circ F_{\Delta k}(R^0_u, R^1_u) \, [ a \otimes \Delta k ]
=
\mul \circ f_{\Delta k}(R^0_u, \Modular_1) \, [ a \otimes \Delta k ]
\\
&= \begin{aligned}[t]
&\mul \circ f_{\Delta k}(R^0_u, \Modular_1) \circ L^1_k \circ g_1(\Modular_1) \, [ a \otimes \Delta h ]\\
&- g^{\mu\nu} \mul \circ f_{\Delta k}(R^0_u, \Modular_1) \circ L^1_k \circ \mul_{12} \circ g_2(\Modular_1, \Modular_2) \, [a \otimes (\dmu h) \otimes (\dnu h) ]
\end{aligned}
\\
&= \begin{aligned}[t]
&\mul \circ R^0_k \circ f_{\Delta k}(R^0_u, \Modular_1) \circ g_1(\Modular_1) \, [ a \otimes \Delta h ]\\
&- g^{\mu\nu} \mul \circ R^0_k \circ f_{\Delta k}(R^0_u, \Modular_1) \circ g_2(\Modular_1, \Modular_2) \, [a \otimes (\dmu h) \otimes (\dnu h) ]
\end{aligned}
\\
&= \begin{aligned}[t]
&2 \mul \circ R^0_k \circ f_{\Delta k}(R^0_u, \Modular_1) \circ g_1(\Modular_1) \, [ a \otimes \Delta \ln k ]\\
&- 4 g^{\mu\nu} \mul \circ R^0_k \circ f_{\Delta k}(R^0_u, \Modular_1 \Modular_2) \circ g_2(\Modular_1, \Modular_2) \, [a \otimes (\dmu \ln k) \otimes (\dnu \ln k) ]
\end{aligned}
\end{align*}
and
\begin{align*}
\mul & \circ F_{\partial k \partial k}^{\mu\nu}(R^0_u, R^1_u, R^2_u) \, [a \otimes (\dmu k) \otimes (\dnu k) ]
= 
\mul \circ f_{\partial k \partial k}^{\mu\nu}(R^0_u, \Modular_1, \Modular_2) \, [a \otimes (\dmu k) \otimes (\dnu k) ]
\\
&=
\mul \circ f_{\partial k \partial k}^{\mu\nu}(R^0_u, \Modular_1, \Modular_2) \circ L^1_k \circ L^2_k \circ g_1(\Modular_1) \circ g_1(\Modular_2) \, [a \otimes (\dmu h) \otimes (\dnu h) ]
\\
&= \mul \circ R^0_u \circ \Modular_1^{1/2} \circ f_{\partial k \partial k}^{\mu\nu}(R^0_u, \Modular_1, \Modular_2) \circ g_1(\Modular_1) \circ g_1(\Modular_2) \, [a \otimes (\dmu h) \otimes (\dnu h) ]
\\
&= 4 \mul \circ R^0_u \circ \Modular_1^{1/2} \circ f_{\partial k \partial k}^{\mu\nu}(R^0_u, \Modular_1, \Modular_2) \circ g_1(\Modular_1) \circ g_1(\Modular_2) \, [a \otimes (\dmu \ln k) \otimes (\dnu \ln k) ].
\end{align*}
So, the sum gives
\begin{align*}
&\mul \circ F_{\Delta k}(R^0_u, R^1_u) \, [ a \otimes \Delta k ]
 + \mul \circ F_{\partial k \partial k}^{\mu\nu}(R^0_u, R^1_u, R^2_u) \, [a \otimes (\dmu k) \otimes (\dnu k) ]
\\
&= \mul \circ G_{(\Delta \ln k)}(R^0_u, \Modular_1)\, [ a \otimes \Delta \ln k ]
+ \mul \circ G^{\mu\nu}_{(\partial \ln k) (\partial \ln k)}(R^0_u, \Modular_1, \Modular_2) \, [a \otimes (\dmu \ln k) \otimes (\dnu \ln k) ]
\end{align*}
with
\begin{align*}
G_{(\Delta \ln k)}(R^0_u, \Modular_1)
&\vc
2 R^0_k \circ f_{\Delta k}(R^0_u, \Modular_1) \circ g_1(\Modular_1),
\\
G^{\mu\nu}_{(\partial \ln k) (\partial \ln k)}(R^0_u, \Modular_1, \Modular_2)
&\vc
\begin{aligned}[t]
&4 R^0_u \circ \Modular_1^{1/2} \circ f_{\partial k \partial k}^{\mu\nu}(R^0_u, \Modular_1, \Modular_2) \circ g_1(\Modular_1) \circ g_1(\Modular_2) 
\\
&- 4 g^{\mu\nu} R^0_k \circ f_{\Delta k}(R^0_u, \Modular_1 \Modular_2) \circ g_2(\Modular_1, \Modular_2).
\end{aligned}
\end{align*}
The associated spectral functions are
\begin{align*}
G_{(\Delta \ln k)}(r_0, y_1)
&=
2 \sqrt{r_0} f_{\Delta k}(r_0, y_1) g_1(y_1),
\\
G^{\mu\nu}_{(\partial \ln k) (\partial \ln k)}(r_0, y_1, y_2)
&=
4 r_0 \sqrt{y_1} f_{\partial k \partial k}^{\mu\nu}(r_0, y_1, y_2) g_1(y_1) g_1(y_2) 
- 4 g^{\mu\nu} \sqrt{r_0} f_{\Delta k}(r_0, y_1 y_2) g_2(y_1, y_2).
\end{align*}

Another change of convention concerns the derivations of $C^\infty(\TT^2_\Theta)$: in \cite{FathKhal11a}, $\hdmu \vc - i \dmu$ is used. This implies that their expressions like $(\hdmu \ln k) (\hdnu \ln k)$ correspond to our $- (\dmu \ln k) \otimes (\dnu \ln k)$. Notice also their combination $\hdelta_1^2 \ln k + \abs{\tau}^2 \hdelta_2^2 \ln k + 2 \tau_1 \tau_2 \hdelta_1 \hdelta_2 \ln k = g^{\mu\nu} \hdmu \hdnu \ln k = - g^{\mu\nu} \dmu \dnu \ln k = \Delta \ln k$. Thus for a comparison of the two results, a $-$ sign has to be taken into account for the $G^{\mu\nu}_{(\partial \ln k) (\partial \ln k)}$ term. Finally, \cite[Theorem~5.2]{FathKhal11a} is written in terms of functions of $\ln \Modular$, thus it remains to make the final change of variables $y_1 = e^x$ in $G_{(\Delta \ln k)}$, and $y_1 = e^s$ and $y_2 = e^t$ in $G^{\mu\nu}_{(\partial \ln k) (\partial \ln k)}$. 

In \cite[Theorem~5.2]{FathKhal11a}, $\mathcal{R}_{FK} = -\tfrac{\pi}{\tau_2} \times [\text{expression in $R_1$, $R_2$, $W$}]$ while  in \eqref{eq-NCT-result}, one has written $\mathcal{R}_2 = \tfrac{1}{4\,\pi} \times \text{[expression in $F_{\Delta k}$, $F_{\partial k \partial k}^{\mu\nu}$]}$. The proof that $\mathcal{R}_2 = \tfrac{\tau_2}{(2 \pi)^2} \mathcal{R}_{FK}$ is then equivalent to check that $[\text{expression in $R_1$, $R_2$, $W$}] = - \text{[expression in $F_{\Delta k}$, $F_{\partial k \partial k}^{\mu\nu}$]}$. The previous technical results imply that this is equivalent to show that
\begin{align*}
R_1(x) 
&= 
- G_{(\Delta \ln k)}(r_0, e^x) ,
\\
R_2(s,t)
&=
G^{11}_{(\partial \ln k) (\partial \ln k)}(r_0, e^s, e^t),
\\
\abs{\tau}^2 R_2(s,t)
&= 
G^{22}_{(\partial \ln k) (\partial \ln k)}(r_0, e^s, e^t),
\\
\tau_1 R_2(s,t)  - i \tau_2 W(s,t) 
 &= 
 G^{12}_{(\partial \ln k) (\partial \ln k)}(r_0, e^s, e^t),
\\
\tau_1 R_2(s,t) + i \tau_2 W(s,t) 
 &= 
 G^{21}_{(\partial \ln k) (\partial \ln k)}(r_0, e^s, e^t).
\end{align*}
All these relations can be checked directly. In particular, the relations on the RHS are independent of the variable $r_0$.

\bigskip
In order to compare \eqref{eq-curv-NCT-4} for the noncommutative four torus with \cite[Theorem~5.4]{FathKhal13a}, we can use the results in \cite{Fath15a}. As before, we need the correspondence \eqref{eq-varphi-NCtr-d=4} between our trace $\varphi$ and their trace $\varphi_0 \equiv \NCtr$. Here $g_i^{\mu\nu} = \delta^{\mu\nu}$ on the base tori $\TT^{2}_{B, i}$, so that $\abs{g_i}^{1/2} = 1$. Denote by $\mathcal{R}_{FK}$ the curvature obtained in \cite[Theorem~5.4]{FathKhal13a}, which is $\pi^2$ times \cite[eq.~(5.1)]{FathKhal13a}. A comparison between eq.~(1) and (3) in \cite{Fath15a} and \cite[eq.~(5.1)]{FathKhal13a} gives
\begin{align*}
\varphi_0(a \mathcal{R}_{FK}) 
&= \tfrac{1}{2} \varphi_0 \left( a \int_{\gS^3} b_2(\xi) d \Omega \right)
= \tfrac{\pi^2}{2} \varphi_0 \left( a [
- \delta^{\mu\nu} k^{-2} (\hdmu \hdnu k^2) k^{-2}  
+ \tfrac{3}{2} \delta^{\mu\nu} k^{-2} (\hdmu k^2) k^{-2} (\hdnu k^2) k^{-2}
]\right)
\\
&= \tfrac{\pi^2}{2 (2\pi)^4} \varphi \left( a [
\delta^{\mu\nu} k^{-2} (\dmu \dnu k^2) k^{-2}  
- \tfrac{3}{2} \delta^{\mu\nu} k^{-2} (\dmu k^2) k^{-2} (\dnu k^2) k^{-2}
]\right)
= \varphi \left( a \mathcal{R}_2 \right),
\end{align*}
and the two results coincide.\footnote{The constant $c$ computed in \cite[eq.~(3)]{Fath15a} is not $c = 1/(2 \pi^2)$ as claimed but $c = 1/2$ as shown by a direct comparison between \cite[eq.~(3)]{Fath15a} and \cite[eq.~(5.1)]{FathKhal13a} where a factor $\pi^2$ is not written. This explains the factor $\tfrac{1}{2}$ in our first equality.}

\section*{References}

\bibliographystyle{elsarticle-num-names}
\bibliography{Heat-trace-2-biblio}

\begin{thebibliography}{28}
\providecommand{\natexlab}[1]{#1}
\providecommand{\url}[1]{\texttt{#1}}
\providecommand{\urlprefix}{URL }
\expandafter\ifx\csname urlstyle\endcsname\relax
  \providecommand{\doi}[1]{doi:\discretionary{}{}{}#1}\else
  \providecommand{\doi}[1]{doi:\discretionary{}{}{}\begingroup
  \urlstyle{rm}\url{#1}\endgroup}\fi
\providecommand{\bibinfo}[2]{#2}

\bibitem[{Iochum and Masson(2017{\natexlab{a}})}]{IochMass17a}
\bibinfo{author}{B.~Iochum}, \bibinfo{author}{T.~Masson}, \bibinfo{title}{Heat
  trace for {L}aplace type operators with non-scalar symbols},
  \bibinfo{journal}{Journal of Geometry and Physics} \bibinfo{volume}{116}
  (\bibinfo{year}{2017}{\natexlab{a}}) \bibinfo{pages}{90--118}.

\bibitem[{Avramidi(2004)}]{Avra04a}
\bibinfo{author}{I.~G. Avramidi}, \bibinfo{title}{Gauged gravity via spectral
  asymptotics of non-Laplace type operators}, \bibinfo{journal}{Journal of High
  Energy Physics} \bibinfo{volume}{2004}~(\bibinfo{number}{07})
  (\bibinfo{year}{2004}) \bibinfo{pages}{030}.

\bibitem[{Avramidi and Branson(2001)}]{AvraBran01a}
\bibinfo{author}{I.~G. Avramidi}, \bibinfo{author}{T.~P. Branson},
  \bibinfo{title}{Heat kernel asymptotics of operators with non-Laplace
  principal part}, \bibinfo{journal}{Reviews in Mathematical Physics}
  \bibinfo{volume}{13}~(\bibinfo{number}{07}) (\bibinfo{year}{2001})
  \bibinfo{pages}{847--890}.

\bibitem[{{Wolfram Research Inc.}(2017)}]{Wolf17a}
\bibinfo{author}{{Wolfram Research Inc.}}, \bibinfo{title}{Mathematica, Version
  11.1}, \bibinfo{year}{2017}.

\bibitem[{Iochum and Masson(2017{\natexlab{b}})}]{IochMass17b}
\bibinfo{author}{B.~Iochum}, \bibinfo{author}{T.~Masson}, \bibinfo{title}{{Heat
  asymptotics for nonminimal Laplace type operators and application to
  noncommutative tori}}, \bibinfo{note}{{Mathematica Notebook added as
  ancillary file on arXiv}}, \bibinfo{year}{2017}{\natexlab{b}}.

\bibitem[{Connes and Tretkoff(2011)}]{ConnTret11a}
\bibinfo{author}{A.~Connes}, \bibinfo{author}{P.~Tretkoff}, \bibinfo{title}{The
  {Gauss--Bonnet} theorem for the noncommutative two torus},
  \bibinfo{journal}{Noncommutative Geometry, Arithmetic, and Related Topics,
  Johns Hopkins Univ. Press, Baltimore, MD}  (\bibinfo{year}{2011})
  \bibinfo{pages}{141--158}.

\bibitem[{Connes and Moscovici(2014)}]{ConnMosc14a}
\bibinfo{author}{A.~Connes}, \bibinfo{author}{H.~Moscovici},
  \bibinfo{title}{Modular curvature for noncommutative two-tori},
  \bibinfo{journal}{Journal of the American Mathematical Society}
  \bibinfo{volume}{27}~(\bibinfo{number}{3}) (\bibinfo{year}{2014})
  \bibinfo{pages}{639--684}.

\bibitem[{Fathizadeh and Khalkhali(2013)}]{FathKhal11a}
\bibinfo{author}{F.~Fathizadeh}, \bibinfo{author}{M.~Khalkhali},
  \bibinfo{title}{Scalar curvature for the noncommutative two torus},
  \bibinfo{journal}{Journal of Noncommutative Geometry} \bibinfo{volume}{7}
  (\bibinfo{year}{2013}) \bibinfo{pages}{1145--1183}.

\bibitem[{Fathizadeh and Khalkhali(2012)}]{FathKhal12a}
\bibinfo{author}{F.~Fathizadeh}, \bibinfo{author}{M.~Khalkhali},
  \bibinfo{title}{The {Gauss-Bonnet} theorem for noncommutative two tori with a
  general conformal structure}, \bibinfo{journal}{Journal of Noncommutative
  Geometry} \bibinfo{volume}{6} (\bibinfo{year}{2012})
  \bibinfo{pages}{457--480}.

\bibitem[{Dabrowski and Sitarz(2013)}]{DabSit13}
\bibinfo{author}{L.~Dabrowski}, \bibinfo{author}{A.~Sitarz},
  \bibinfo{title}{Curved noncommutative torus and Gauss--Bonnet},
  \bibinfo{journal}{Journal of Mathematical Physics} \bibinfo{volume}{54}
  (\bibinfo{year}{2013}) \bibinfo{pages}{013518}.

\bibitem[{Fathizadeh and Khalkhali(2015)}]{FathKhal13a}
\bibinfo{author}{F.~Fathizadeh}, \bibinfo{author}{M.~Khalkhali},
  \bibinfo{title}{Scalar curvature for noncommutative four-tori},
  \bibinfo{journal}{Journal of Noncommutative Geometry} \bibinfo{volume}{9}
  (\bibinfo{year}{2015}) \bibinfo{pages}{473--503}.

\bibitem[{Azzali et~al.(2014)Azzali, L\'evy, Neira-Jim\'enez, and
  Paycha}]{ALNP}
\bibinfo{author}{S.~Azzali}, \bibinfo{author}{C.~L\'evy},
  \bibinfo{author}{C.~Neira-Jim\'enez}, \bibinfo{author}{S.~Paycha},
  \bibinfo{title}{Traces of holomorphic families of operators on the
  noncommutative torus and on Hilbert Modules}, in:
  \bibinfo{booktitle}{Geometrics methods in Physics},
  \bibinfo{publisher}{Birkh\"auser}, \bibinfo{pages}{3--38},
  \bibinfo{year}{2014}.

\bibitem[{Sitarz(2014)}]{Sitarz14}
\bibinfo{author}{A.~Sitarz}, \bibinfo{title}{{Wodzicki residue and minimal
  operators on a noncommutative 4-dimensional torus}}, \bibinfo{journal}{J.
  Pseudo-Differ. Oper. Appl.} \bibinfo{volume}{5} (\bibinfo{year}{2014})
  \bibinfo{pages}{305--317}.

\bibitem[{Fathizadeh(2015)}]{Fath15a}
\bibinfo{author}{F.~Fathizadeh}, \bibinfo{title}{On the scalar curvature for
  the noncommutative four torus}, \bibinfo{journal}{Journal of Mathematical
  Physics} \bibinfo{volume}{56}~(\bibinfo{number}{6}) (\bibinfo{year}{2015})
  \bibinfo{pages}{062303}.

\bibitem[{Dabrowski and Sitarz(2015)}]{DabSit15}
\bibinfo{author}{L.~Dabrowski}, \bibinfo{author}{A.~Sitarz}, \bibinfo{title}{An
  asymmetric noncommutative torus}, \bibinfo{journal}{SIGMA}
  \bibinfo{volume}{11} (\bibinfo{year}{2015}) \bibinfo{pages}{075, 11 pages}.

\bibitem[{Liu(2015)}]{Liu15a}
\bibinfo{author}{Y.~Liu}, \bibinfo{title}{Modular curvature for toric
  noncommutative manifolds}, \bibinfo{note}{arXiv:1510.04668v2 [math.OA]},
  \bibinfo{year}{2015}.

\bibitem[{Sadeghi(2016)}]{Sade16a}
\bibinfo{author}{S.~Sadeghi}, \bibinfo{title}{{On logarithmic Sobolev
  inequality and a scalar curvature formula for noncommutative tori}}, Ph.D.
  thesis, \bibinfo{school}{Western University, Ontario}, \bibinfo{year}{2016}.

\bibitem[{Connes and Fathizadeh(2016)}]{ConnesFath16}
\bibinfo{author}{A.~Connes}, \bibinfo{author}{F.~Fathizadeh},
  \bibinfo{title}{The term $a_4$ in the heat kernel expansion of noncommutative
  tori}, \bibinfo{note}{arXiv:1611.09815v1 [math.QA]}, \bibinfo{year}{2016}.

\bibitem[{Connes(1980)}]{Connes80}
\bibinfo{author}{A.~Connes}, \bibinfo{title}{{C$^*$-alg\`ebres et g\'eom\'etrie
  diff\' erentielle}}, \bibinfo{journal}{C. R. Acad. Sc. Paris - S\'erie A}
  \bibinfo{volume}{290} (\bibinfo{year}{1980}) \bibinfo{pages}{599--604}.

\bibitem[{Lesch and Moscovici(2016)}]{LeschMosco16}
\bibinfo{author}{M.~Lesch}, \bibinfo{author}{H.~Moscovici},
  \bibinfo{title}{Modular curvature and Morita equivalence},
  \bibinfo{journal}{Geom. Funct. Anal.} \bibinfo{volume}{26}
  (\bibinfo{year}{2016}) \bibinfo{pages}{818--873}.

\bibitem[{Pedersen(1976)}]{Pede76a}
\bibinfo{author}{G.~K. Pedersen}, \bibinfo{title}{On the Operator Equation $HT+
  TH= 2K$}, \bibinfo{journal}{Indiana University Mathematics Journal}
  \bibinfo{volume}{25}~(\bibinfo{number}{11}) (\bibinfo{year}{1976})
  \bibinfo{pages}{1029--1033}.

\bibitem[{Gilkey(1995)}]{Gilk95a}
\bibinfo{author}{P.~B. Gilkey}, \bibinfo{title}{Invariance theory, the heat
  equation and the Atiyah-Singer index theorem}, \bibinfo{publisher}{Studies in
  Advanced Mathematics, CRC Press, Inc}, \bibinfo{edition}{2} edn.,
  \bibinfo{year}{1995}.

\bibitem[{Gilkey(2003)}]{Gilk03a}
\bibinfo{author}{P.~B. Gilkey}, \bibinfo{title}{Asymptotic formulae in spectral
  geometry}, \bibinfo{publisher}{CRC press}, \bibinfo{year}{2003}.

\bibitem[{Fell(1961)}]{Fell61a}
\bibinfo{author}{J.~M.~G. Fell}, \bibinfo{title}{The structure of algebras of
  operator fields}, \bibinfo{journal}{Acta Mathematica}
  \bibinfo{volume}{106}~(\bibinfo{number}{3-4}) (\bibinfo{year}{1961})
  \bibinfo{pages}{233--280}.

\bibitem[{Blackadar(2006)}]{Blac06a}
\bibinfo{author}{B.~Blackadar}, \bibinfo{title}{Operator algebras: theory of
  $C^*$-algebras and von Neumann algebras}, vol. \bibinfo{volume}{122},
  \bibinfo{publisher}{Springer Science}, \bibinfo{year}{2006}.

\bibitem[{Gracia-Bond{\'\i}a et~al.(2001)Gracia-Bond{\'\i}a, V{\'a}rilly, and
  Figueroa}]{GracVariFigu01a}
\bibinfo{author}{J.~M. Gracia-Bond{\'\i}a}, \bibinfo{author}{J.~C.
  V{\'a}rilly}, \bibinfo{author}{H.~Figueroa}, \bibinfo{title}{Elements of
  Noncommutative Geometry}, \bibinfo{publisher}{Birkh{\"a}user Boston},
  \bibinfo{year}{2001}.

\bibitem[{Dubois-Violette et~al.(2002)Dubois-Violette, Kriegl, Maeda, and
  Michor}]{DuboKrieMaedMich02a}
\bibinfo{author}{M.~Dubois-Violette}, \bibinfo{author}{A.~Kriegl},
  \bibinfo{author}{Y.~Maeda}, \bibinfo{author}{P.~W. Michor},
  \bibinfo{title}{Smooth *-algebras}, \bibinfo{journal}{Progress of Theoretical
  Physics Supplement} \bibinfo{volume}{144} (\bibinfo{year}{2002})
  \bibinfo{pages}{54--78}.

\bibitem[{Lesch(2017)}]{Lesc14a}
\bibinfo{author}{M.~Lesch}, \bibinfo{title}{{Divided differences in
  noncommutative geometry: Rearrangement lemma, functional calculus and
  expansional formula}}, \bibinfo{journal}{Journal of Noncommutative Geometry}
  \bibinfo{volume}{11} (\bibinfo{year}{2017}) \bibinfo{pages}{193--223}.

\end{thebibliography}

\end{document}